\documentclass[12pt]{amsart}
\usepackage{amssymb,latexsym}
\usepackage{amsfonts}
\usepackage{amsmath}
\usepackage[colorlinks,linkcolor=blue,anchorcolor=blue,citecolor=blue]{hyperref}
\usepackage{algorithm}
\usepackage{enumerate}
\usepackage{algpseudocode}
\usepackage{verbatim}
\usepackage{graphicx}

\newcommand\F{{\mathbb F}}
\newcommand\N{{\mathbb N}}

\newcommand\cG{{\mathcal G}}
\newcommand\cV{{\mathcal V}}

\newtheorem{theorem}{Theorem}[section]

\newtheorem{corollary}[theorem]{Corollary}
\newtheorem{proposition}[theorem]{Proposition}
\newtheorem{conjecture}[theorem]{Conjecture}

\theoremstyle{definition}
\newtheorem{definition}[theorem]{Definition}
\newtheorem{remark}[theorem]{Remark}

\numberwithin{equation}{section}

\begin{document}

\title{On the equational graphs over finite fields}

\author[B. Mans]{Bernard Mans}
\address{Department of Computing, Macquarie University, Sydney, NSW 2109, Australia}
\email{bernard.mans@mq.edu.au}

\author[M. Sha]{Min Sha}
\address{School of Mathematics and Statistics, University of New South Wales, Sydney, NSW 2052, Australia}
\email{shamin2010@gmail.com}

\author[J. Smith]{Jeffrey Smith}
\address{Department of Computing, Macquarie University, Sydney, NSW 2109, Australia}
\email{jeffrey.papworth.smith@gmail.com}

\author[D. Sutantyo]{Daniel Sutantyo}
\address{Department of Computing, Macquarie University, Sydney, NSW 2109, Australia}
\email{daniel.sutantyo@gmail.com}

\keywords{Finite field, functional graph, equational graph, strong connectedness, connected component,  Hamiltonian cycle}

\subjclass[2010]{05C20, 05C38, 05C40, 11T06}



\begin{abstract} 
In this paper, we generalize the notion of functional graph. 
Specifically, given an equation $E(X,Y) = 0$
with variables $X$ and $Y$
over a finite field $\F_q$ of odd characteristic, 
we define a digraph by choosing the elements in $\F_q$ as vertices and drawing an edge from $x$ to $y$ if and only if $E(x,y)=0$. 
We call this graph as equational graph. 
In this paper, we study the equational graphs when choosing $E(X,Y) = (Y^2 - f(X))(\lambda Y^2 - f(X))$ 
with $f(X)$ a polynomial over $\F_q$ and  $\lambda$ a non-square element in $\F_q$. 
We show that if $f$ is a permutation polynomial over $\F_q$, then every connected component of the graph has a Hamiltonian cycle.
Moreover, these Hamiltonian cycles can be used to construct balancing binary sequences.
By making computations for permutation polynomials $f$ of low degree, it appears that almost all these graphs are strongly connected,
and there are many Hamiltonian cycles in such a graph if it is connected.
\end{abstract}

\maketitle

\section{Introduction}

Let $\F_q$ be the finite field of $q$ elements, where $q$ is a power of some odd prime $p$.
Let $\F_q^*$ be the set of non-zero elements in $\F_q$.
For any polynomial $f \in \F_q[X]$,
we define the \textit{functional graph} of $f$ as a digraph on $q$ vertices
labelled by the elements of $\F_q$,  where
there is an edge from $x$ to $y$ if and only if $f(x) = y$.
These graphs have been extensively studied in recent years; see \cite{BuSch,FlGar,HS,KLMMSS,MSSS,OstSha,VaSha} and the references therein. 
The motivation for studying these graphs comes from several resources, 
such as Lucas-Lehmer primality test for  Mersenne numbers, Pollard's rho algorithm for integer factorization, pseudo-random number generators, 
and arithmetic dynamics.

In the above construction, we in fact use the equation $Y - f(X) = 0$.
Then, there is an edge from $x$ to $y$ if and only if $y-f(x)=0$.
Hence, more generally, for any equation over $\F_q$: 
$$
E(X,Y) = 0
$$
with variables $X$ and $Y$, we define a digraph by choosing the elements in $\F_q$ as vertices and drawing an edge from $x$ to $y$ if and only if
$E(x,y) = 0$. We call this graph an \textit{equational graph} of the above equation. 
However, it might happen that this equation has no solution over $\F_q$. 

We remark that clearly functional graph and equational graph can be defined similarly over finite fields of even characteristic. 
But in this paper we only consider finite fields of odd characteristic. 

In this paper, we consider equational graphs generated by equations of the form
\begin{equation}   \label{eq:equation}
(Y^2 - f(X))(\lambda Y^2 - f(X)) = 0
\end{equation} 
with variables $X$ and $Y$, 
where $f(X)$ is a fixed polynomial over $\F_q$ and  $\lambda$ is a fixed non-square element in $\F_q$. 
Then, there is an edge from $x$ to $y$ if and only if $(y^2 - f(x))(\lambda y^2 - f(x)) = 0$.
This yields an equational graph over $\F_q$, denoted by $\cG(\lambda,f)$; see Figure~\ref{fig:linear} for a simple example.
Since $\lambda$ is non-square, the out-degree of each vertex $x$ is positive, which in fact equals to $2$ if $f(x) \ne 0$ (because if there is an edge from $x$ to $y$,
then there is also an edge from $x$ to $-y$).
However, the in-degree of $x$ can be from zero to the degree of $f$. 
Note that we allow the graph $\cG(\lambda,f)$ to have loops. 

\begin{figure}[!htbp]
\begin{center}
\setlength{\unitlength}{1cm}
\begin{picture}(20,3.5)
\multiput(2.5,0)(2,0){2}{\circle{0.7}}
\multiput(2.5,3.2)(2,0){2}{\circle{0.7}}
\multiput(6.5,1.6)(2,0){3}{\circle{0.7}}

\put(2.35, 2.875){\vector(0,-1){2.55}}
\put(2.65, 0.325){\vector(0,1){2.55}}
\put(2.85, 0){\vector(1,0){1.3}}
\put(4.15, 3.2){\vector(-1,0){1.3}}
\put(4.5, 2.85){\vector(0,-1){2.5}}
\put(4.75, 0.25){\vector(3,2){1.57}}
\put(6.3, 1.9){\vector(-3,2){1.55}}
\put(6.835, 1.5){\vector(1,0){1.33}}
\put(8.165, 1.7){\vector(-1,0){1.33}}
\put(8.835, 1.5){\vector(1,0){1.33}}
\put(10.165, 1.7){\vector(-1,0){1.33}}
\put(4.85, 0){\vector(4,1){5.4}}
\put(10.3, 1.9){\vector(-4,1){5.45}}

\put(2.4,-0.15){$0$}
\put(4.4,-0.15){$1$}
\put(2.4,3.05){$6$}
\put(4.4,3.05){$2$}
\put(6.4,1.45){$4$}
\put(8.4,1.45){$5$}
\put(10.4,1.45){$3$}
\end{picture}
\end{center}
\caption{The equational graph $\cG(3,X+1)$ over $\F_7$}
\label{fig:linear}
\end{figure}
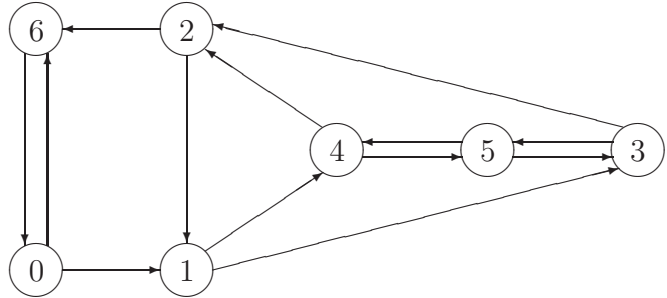

We show that if  $f$ is a permutation polynomial over $\F_q$,
then every (weakly) connected component of the graph $\cG(\lambda,f)$ is strongly connected (see Proposition~\ref{prop:perm-conn})
 and has a Hamiltonian cycle (see Theorem~\ref{thm:perm-Ha}).
By distinguishing the edges according to the subequations ($Y^2 - f(X)=0$ or $\lambda Y^2 - f(X)=0$) they come from,
we classify these Hamiltonian cycles (see Definition~\ref{def:type}) and show that there is no Hamiltonian cycle of Type 1 for many such graphs
(see Theorem~\ref{thm:HT1}, Corollary~\ref{cor:linear-HT1} and Theorem~\ref{thm:perm-HT1}).
Moreover, we prove that these Hamiltonian cycles can be used to construct balancing binary sequences by associating weights to the edges
(see Theorem~\ref{thm:perm-balance}).

Through making computations for permutation polynomials $f$ of low degree, it appears that almost all these graphs are strongly connected,
and each such connected graph has many Hamiltonian cycles.
That is, using these graphs we can frequently obtain balancing periodic sequences of period $q$.
Additionally, we also investigate the graphs $\cG(\lambda,f)$ with polynomials $f$ of low degree in more detail.

We remark that by construction, any graph $\cG(\lambda,f)$ with permutation polynomial $f$ 
is quite close to be a 2-regular digraph. 
The result in \cite[Theorem 5.1]{FF} implies that  almost every 2-regular digraph is strongly connected, 
and then the result in \cite[Theorem 1]{CF} suggests that almost every 2-regular digraph has a Hamiltonian cycle. 
We thus can view such graphs $\cG(\lambda,f)$ as typical examples for this. 

The paper is organized as follows:
Section~\ref{sec:perm} deals with the case when $f$ is a permutation polynomial over $\F_q$,
and the algorithms for the computations of its connectedness and Hamiltonian cycles are presented in Section~\ref{sec:alg}.
We then study the cases when $f$ is of degree 1, 2 and 3 in Sections~\ref{sec:linear}, \ref{sec:quad} and \ref{sec:cubic} respectively.
Finally we make some comments for further study.

\section{The case of permutation polynomials}
\label{sec:perm}

Here, in the graph $\cG(\lambda,f)$ we choose $f$ to be a permutation polynomial over $\F_q$.
That is, the map $x \mapsto f(x)$ is a bijection from $\F_q$ to itself.
We refer to \cite[Chapter 7]{LN} for an extensive introduction on permutation polynomials.

Recall that $q$ is odd, and $\lambda$ is a non-square element in $\F_q$.

\subsection{Basic properties}

First, it is easy to determine the in-degrees and out-degrees of the vertices in the graph $\cG(\lambda,f)$.

\begin{proposition}
\label{prop:perm-inout}
Let $f$ be a permutation polynomial over $\F_q$.
If $f(0) \ne 0$, then in the graph $\cG(\lambda,f)$,  the vertex $0$ has in-degree $1$ and out-degree $2$,
the vertex   $f^{-1}(0)$ has in-degree $2$ and out-degree $1$,
and any other vertex $x$ $(x \ne 0$ and $x \ne f^{-1}(0))$ has in-degree $2$ and out-degree $2$.
Otherwise, if $f(0) = 0$, then in the graph $\cG(\lambda,f)$,  the vertex $0$ has in-degree $1$ and out-degree $1$,
 and any other vertex $x$ $(x \ne 0)$ has in-degree $2$ and out-degree $2$.
\end{proposition}

\begin{proof}
The proof is quite straightforward.
We only need to note that the map $x \mapsto f(x)$ gives a bijection from $\F_q$ to itself.
\end{proof}

Moreover,  each (weakly) connected component of the graph $\cG(\lambda,f)$ is strongly connected.

\begin{proposition}
\label{prop:perm-conn}
In each graph $\cG(\lambda,f)$ with permutation polynomial $f$, every vertex lies in a cycle, and also every edge lies in a cycle. 
In particular, every connected component is strongly connected.
\end{proposition}

\begin{proof}
Let $C$ be an arbitrary connected component of the graph $\cG(\lambda,f)$.
For our purpose, it suffices to show that given an edge from $x$ to $y$ in $C$, there is a (directed) path from $y$ to $x$.
Notice that by Proposition~\ref{prop:perm-inout} the in-degree and out-degree of each vertex in the graph $\cG(\lambda,f)$ are both positive.
Then, starting from $x$, we draw the predecessors of $x$, and the predecessors of the predecessors of $x$, and so on; this gives a subgraph, say $G_1$.
While starting from $y$, we draw the successors of $y$, and the successors of the successors of $y$, and so on; this gives another subgraph, say $G_2$.
If $G_1$ and $G_2$ have a common vertex, then everything is done.
So, we only need to prove that these two subgraphs indeed have a common vertex.

Now, by contradiction, suppose that $G_1$ and $G_2$ have no common vertex.
Then, the edge from $x$ to $y$ is not in $G_1$ and also not in $G_2$. 
Without loss of generality, we can assume that $G_2$ does not contain the vertex $0$.
So, noticing $y \ne 0$ and $y \ne f^{-1}(0)$ and using Proposition~\ref{prop:perm-inout}, 
in $G_2$ the vertex $y$ has out-degree $2$ and in-degree at most $1$,
and any other vertex in $G_2$ has out-degree $2$ and  in-degree at most $2$.
Thus, the sum of out-degrees in $G_2$ is greater than the sum of  in-degrees. 
But in fact they must be equal. 
Hence, $G_1$ and $G_2$ indeed have a common vertex.
\end{proof}

The following proposition suggests that the graph $\cG(\lambda,f)$ can be complicated.

\begin{proposition}
\label{prop:perm-bi}
Each graph $\cG(\lambda,f)$ with permutation polynomial $f$ is not a bipartite graph.
\end{proposition}

\begin{proof}
By contradiction, assume that the graph $\cG(\lambda,f)$ over $\F_q$ is a bipartite graph.
Then, the vertex set can be separated into two subsets, say $S_1$ and $S_2$, such that
there are no edges among the vertices in $S_1$ and also there are no edges among the vertices in $S_2$. 
So, there are no loops in $\cG(\lambda,f)$, which implies $f(0) \ne 0$. 
Without loss of generality, we assume that $f^{-1}(0) \in S_1$ and $0 \in S_2$.
Let $m=|S_1|$ and $n=|S_2|$.
Then, by Proposition~\ref{prop:perm-inout}, the sum of out-degrees of the vertices in $S_1$ is equal to $2(m-1)+1$,
and the sum of in-degrees of the vertices in $S_2$ is equal to $2(n-1)+1$.
By assumption, we must have
$$
2(m-1)+1 = 2(n-1)+1,
$$
which implies that $m=n$, and so $m+n$ is an even integer.
However, $m+n =q$, and $q$ is odd.
This leads to a contradiction.
So, $\cG(\lambda,f)$ is not a bipartite graph.
\end{proof}

\subsection{Existence of Hamiltonian cycles}

For each graph $\cG(\lambda,f)$ over $\F_q$ with permutation polynomial $f$,
by Propositions~\ref{prop:perm-inout} and \ref{prop:perm-conn} its connected components are close to be 
strongly connected $2$-regular digraphs except at the vertex $0$ when $f(0)=0$.
Note that not every strongly connected $2$-regular digraph (even without loops) has a Hamiltonian cycle; 
see, for example,  \cite[Corollary 3.8.2]{GR}.
However, for the graph $\cG(\lambda,f)$, its connected components all have Hamiltonian cycles.

\begin{theorem}
\label{thm:perm-Ha}
In each graph $\cG(\lambda,f)$ over $\F_q$ with permutation polynomial $f$,  every connected component has a Hamiltonian cycle.
\end{theorem}

\begin{proof}
Let $C$ be an arbitrary connected component of the graph $\cG(\lambda,f)$.
By contradiction, suppose that $C$ has no Hamiltonian cycle.
Then, by Proposition~\ref{prop:perm-conn}, we choose a maximal cycle, say $M$, in $C$ such that
if the vertex $0$ lies in $C$, then $0$ also lies in $M$.
Note that by the maximal assumption, the cycle $M$ can not be  enlarged.

First, since $C$ has no Hamiltonian cycle, the cycle $M$ does not go through all the vertices of $C$.
Then, at least one of the vertices in $M$ has a successor not in $M$; see Figure~\ref{fig:perm-Ha1}.
In Figure~\ref{fig:perm-Ha1}, the vertex $y_0$ is a successor of the vertex $x_0$ and is outside of the cycle $M$.
Note that if the vertex $0$ lies in $C$, then it also lies in $M$.
So, we have $y_0 \ne 0$, and thus the in-degree of $y_0$ is $2$.
By construction and noticing that the out-degree of each vertex is at most two,
the vertex $y_0$ must have a predecessor  not in $M$, say $z_0$  (because there is an edge from $z_0$ to $-y_0$).
In fact, either $x_0=f^{-1}(y_0^2), z_0 = f^{-1}(\lambda y_0^2) $, or $x_0=f^{-1}(\lambda y_0^2), z_0 = f^{-1}(y_0^2)$.

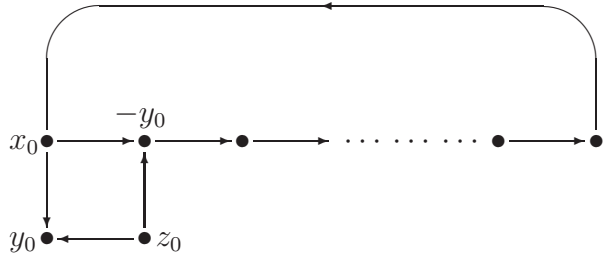
\begin{figure}[!htbp]
\begin{center}
\setlength{\unitlength}{1cm}
\begin{picture}(20,3)
\multiput(2.5,1)(1.3,0){3}{$\bullet$}
\multiput(2.75, 1.1)(1.3,0){3}{\vector(1,0){1}}
\multiput(6.55,1)(0.65,0){3}{$\cdots$}
\multiput(8.5,1)(1.3,0){2}{$\bullet$}
\put(8.75, 1.1){\vector(1,0){1}}
\put(6.25,1.25){\oval(7.3,3.3)[t]}
\put(6.25, 2.9){\vector(-1,0){0}}

\multiput(2.5,-0.3)(1.3,0){2}{$\bullet$}
\put(2.6, 0.95){\vector(0,-1){1}}
\put(3.9, -0.05){\vector(0,1){1}}
\put(3.75, -0.2){\vector(-1,0){1}}
\put(2.1,1){$x_0$}
\put(3.5,1.35){$-y_0$}
\put(2.1,-0.3){$y_0$}
\put(4.05,-0.3){$z_0$}
\end{picture}
\end{center}
\caption{The cycle $M$}
\label{fig:perm-Ha1}
\end{figure}

Now, in Figure~\ref{fig:perm-Ha1}, by Proposition~\ref{prop:perm-conn}, the edge from $z_0$ to $y_0$
lies in a cycle, say $C_1$. If this cycle does not intersect with $M$, then we can emerge the cycles $M$ and $C_1$
by dropping the edge from $x_0$ to $-y_0$ and the edge from $z_0$ to $y_0$.
This gives a larger cycle, but this contradicts with the maximal assumption on $M$.
So, the cycle $C_1$ must intersect with $M$.
Then, there must exist a vertex, say $x_1$, in $C_1$ and also in $M$ such that along the cycle $C_1$ the path from $x_1$ to $y_0$
does not intersect with $M$ except the vertex $x_1$; see Figure~\ref{fig:perm-Ha2} for example.
As the above, $y_1 \ne 0$, and the vertex $z_1$ does not lie in $M$.

\begin{figure}[!htbp]
\begin{center}
\setlength{\unitlength}{1cm}
\begin{picture}(20,3)
\multiput(2.5,1)(1.3,0){2}{$\bullet$}
\multiput(2.75, 1.1)(1.3,0){2}{\vector(1,0){1}}
\multiput(5.2,1)(0.65,0){2}{$\cdots$}
\multiput(6.5,1)(1.3,0){2}{$\bullet$}
\put(6.75, 1.1){\vector(1,0){1}}
\put(8.05, 1.1){\vector(1,0){1}}
\multiput(9.25,1)(0.65,0){2}{$\cdots$}
\multiput(10.5,1)(1.3,0){2}{$\bullet$}
\put(10.75, 1.1){\vector(1,0){1}}
\put(7.25,1.25){\oval(9.3,3.3)[t]}
\put(7.25, 2.9){\vector(-1,0){0}}

\multiput(2.5,-0.3)(1.3,0){2}{$\bullet$}
\put(2.6, 0.95){\vector(0,-1){1}}
\put(3.9, -0.05){\vector(0,1){1}}
\put(3.75, -0.2){\vector(-1,0){1}}
\put(2.1,1){$x_0$}
\put(3.5,1.35){$-y_0$}
\put(2.1,-0.3){$y_0$}
\put(4.05,0){$z_0$}

\multiput(6.5,-0.3)(1.3,0){2}{$\bullet$}
\put(6.6, 0.95){\vector(0,-1){1}}
\put(7.9, -0.05){\vector(0,1){1}}
\put(7.75, -0.2){\vector(-1,0){1}}
\put(6.45, -0.2){\vector(-1,0){1}}
\multiput(4.15,-0.3)(0.65,0){2}{$\cdots$}
\put(6.5,1.35){$x_1$}
\put(7.5,1.35){$-y_1$}
\put(6.75,0){$y_1$}
\put(8.05,0){$z_1$}
\end{picture}
\end{center}
\caption{Going through the procedure}
\label{fig:perm-Ha2}
\end{figure}
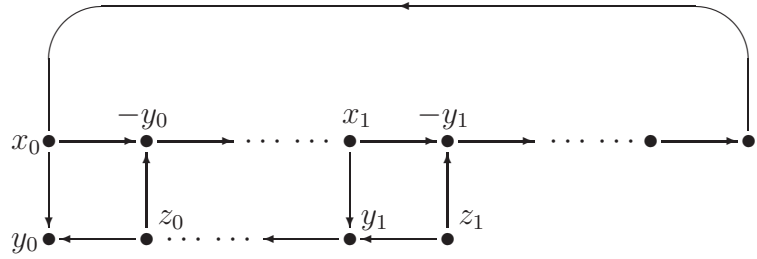

We then go through the above procedure again and again, and thus we obtain an infinite sequence of vertices in $M$: $x_0, x_1, x_2, \ldots$. 
For example, in Figure~\ref{fig:perm-Ha2}, by Proposition~\ref{prop:perm-conn}, the edge from $z_1$ to $y_1$
lies in a cycle, say $C_2$.
As the above, the cycle $C_2$ must intersect with $M$,
and there must exist a vertex, say $x_2$, in $C_2$ and also in $M$ such that along the cycle $C_2$ the path from $x_2$ to $y_1$
does not intersect with $M$ except the vertex $x_2$.
Then, we draw vertices $y_2,-y_2,z_2$ and the edges among them as before (note that $y_2 \ne 0$, and $z_2$ does not lie in $M$).

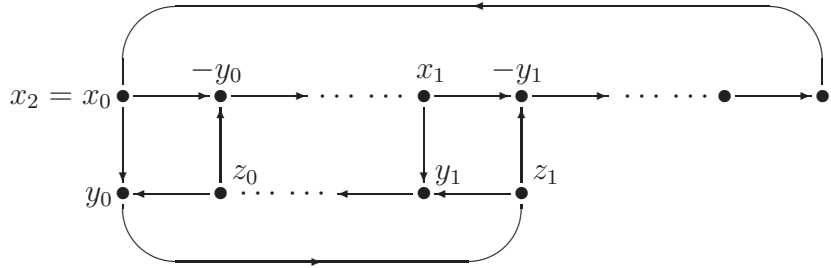
\begin{figure}[!htbp]
\begin{center}
\setlength{\unitlength}{1cm}
\begin{picture}(20,3.5)
\multiput(2.5,1.8)(1.3,0){2}{$\bullet$}
\multiput(2.75, 1.9)(1.3,0){2}{\vector(1,0){1}}
\multiput(5.2,1.8)(0.65,0){2}{$\cdots$}
\multiput(6.5,1.8)(1.3,0){2}{$\bullet$}
\put(6.75, 1.9){\vector(1,0){1}}
\put(8.05, 1.9){\vector(1,0){1}}
\multiput(9.25,1.8)(0.65,0){2}{$\cdots$}
\multiput(10.5,1.8)(1.3,0){2}{$\bullet$}
\put(10.75, 1.9){\vector(1,0){1}}
\put(7.25,2.05){\oval(9.3,2.1)[t]}
\put(7.25, 3.1){\vector(-1,0){0}}

\multiput(2.5,0.5)(1.3,0){2}{$\bullet$}
\put(2.6, 1.75){\vector(0,-1){1}}
\put(3.9, 0.75){\vector(0,1){1}}
\put(3.75, 0.6){\vector(-1,0){1}}
\put(1.1,1.8){$x_2=x_0$}
\put(3.5,2.15){$-y_0$}
\put(2.1,0.5){$y_0$}
\put(4.05,0.8){$z_0$}

\multiput(6.5,0.5)(1.3,0){2}{$\bullet$}
\put(6.6, 1.75){\vector(0,-1){1}}
\put(7.9, 0.75){\vector(0,1){1}}
\put(7.75, 0.6){\vector(-1,0){1}}
\put(6.45, 0.6){\vector(-1,0){1}}
\multiput(4.15,0.5)(0.65,0){2}{$\cdots$}
\put(6.5,2.15){$x_1$}
\put(7.5,2.15){$-y_1$}
\put(6.75,0.8){$y_1$}
\put(8.05,0.8){$z_1$}
\put(5.25,0.45){\oval(5.3,1.5)[b]}
\put(5.25, -0.3){\vector(1,0){0}}
\end{picture}
\end{center}
\caption{The case $x_2=x_0$}
\label{fig:perm-Ha3}
\end{figure}

Since $M$ is a finite cycle, we must have $x_i = x_j$ for some integers $i, j \ge 0$. 
Without loss of generality, we assume  $x_2=x_0$. Then, the picture looks like Figure~\ref{fig:perm-Ha3}. 
By going through the edges from $x_1$ to $y_1$, to $z_0$, to $y_0$, to $z_1$ and then to $-y_1$, we can enlarge the cycle $M$.
This contradicts with the maximal assumption on $M$. 
Therefore, the connected component $C$ indeed has a Hamiltonian cycle.
\end{proof}

We remark that  the result in Theorem~\ref{thm:perm-Ha} can hold
 for some non-permutation polynomials, such as $f(X)=X^2$.

\begin{remark}
For a connected graph $\cG(\lambda,f)$ over $\F_q$,
by Theorem~\ref{thm:perm-Ha} there is a Hamiltonian cycle travelling through all the $q$ vertices,
and then outputing the vertices along the Hamiltonian cycle can give a pseudo-random number generator.
\end{remark}

\subsection{Classification of Hamiltonian cycles}

Recall that the edges of a graph $\cG(\lambda,f)$ come from either $Y^2=f(X)$ or $\lambda Y^2 =  f(X)$.
Using this we can classify the Hamiltonian cycles of connected components of $\cG(\lambda,f)$. 

We first associate weights to the edges in $\cG(\lambda,f)$. 

\begin{definition}  \label{def:weight}
For any edge $(x, y)$ in $\cG(\lambda,f)$, if the edge comes from the relation $y^2 = f(x)$, then its weight is $0$, and otherwise its weight is $1$. 
In particular, the edge going to the vertex $0$ has weight $0$. 
\end{definition}

We now can classify the (directed) paths and Hamiltonian cycles in $\cG(\lambda,f)$. 

\begin{definition}  \label{def:type}
A trail  in $\cG(\lambda,f)$ is a path with all the edges of the same weight. 
A path in $\cG(\lambda,f)$ is said to be of Type $n$ ($n$ is a positive integer)   
if it contains a trail of length $n$ but it contains no trail of length greater than $n$. 
Then, a Hamiltonian cycle $H$ of a connected component in  $\cG(\lambda,f)$ is said to be of Type $n$ 
if $H \setminus \{0\}$ is a path of Type $n$. 
\end{definition}

In Definition \ref{def:type}, we exclude the edge going to the vertex $0$,
because it can be viewed from both $Y^2=f(X)$ and $\lambda Y^2=  f(X)$.

For any polynomial $f \in \F_q[X]$, denote by $\cV(f)$ the value set of $f$, that is,
$$
\cV(f) = \{f(a): \, a\in \F_q\}.
$$
We now want to find a large class of connected graphs $\cG(\lambda,f)$ which do not have
Hamiltonian cycles of Type 1.

\begin{theorem}  \label{thm:HT1}
Let $f \in \F_q[X]$ be a permutation polynomial.
Suppose that  the graph $\cG(\lambda,f)$ is connected,
$$
|\{(f^{-1}(a^2))^2:\,  a \in \F_q\}| \ne \frac{q-1}{2}
$$
and
$$
|\{(f^{-1}(\lambda a^2))^2:\,  a \in \F_q\}| \ne \frac{q-1}{2}
$$
Then,  the graph $\cG(\lambda,f)$ has no Hamiltonian cycle of Type 1.
\end{theorem}

\begin{proof}
Since $f \in \F_q[X]$ is a permutation polynomial,
there  is a permutation polynomial $g \in \F_q[X]$ such that
both $f(g(X))$ and $g(f(X))$ induce the identity map from $\F_q$ to itself.
By assumption,
\begin{equation}  \label{eq:Vg1}
|\cV (g(X^2)^2)| \ne \frac{q-1}{2}, \qquad |\cV(g(\lambda X^2)^2)| \ne \frac{q-1}{2}.
\end{equation}

Since the graph $\cG(\lambda,f)$ is connected, by Theorem~\ref{thm:perm-Ha} it has a Hamiltonian cycle.
Let $H$ be an arbitrary Hamiltonian cycle of $\cG(\lambda,f)$.
We prove the desired result by contradiction.
Suppose that $H$ is of Type 1.
Then, along the cycle $H$, we obtain a path $P$ of Type 1 containing $q$ vertices and from the vertex $0$ to the vertex $g(0)$.
So, there are no two consecutive edges in $P$  having the same weight.

Clearly, $|\cV (g(X^2)^2)| \le (q+1)/2$ and $|\cV (g(\lambda X^2)^2)| \le (q+1)/2$.
If either $|\cV (g(X^2)^2)| = (q+1)/2$ or $|\cV (g(\lambda X^2)^2)| = (q+1)/2$,
then in view of $\cV (g(X^2)) \cap \cV (g(\lambda X^2))=\{g(0)\}$
and $\cV (g(X^2)) \cup \cV (g(\lambda X^2))=\F_q$, we must have $g(0)=0$.
This contradicts the fact $g(0) \ne 0$.
So, noticing \eqref{eq:Vg1} we must have
\begin{equation}  \label{eq:Vg2}
|\cV (g(X^2)^2)| \le \frac{q-3}{2}, \qquad |\cV(g(\lambda X^2)^2)| \le \frac{q-3}{2}.
\end{equation}

Note that by reversing the directions of all the edges in the graph $\cG(\lambda,f)$, 
we can see that this exactly gives the equational graph, say $\cG^\prime (\lambda,g)$, generated by the equation 
$$
(Y-g(X^2))(Y-g(\lambda X^2))=0.
$$
Then, the path $P$ in $\cG(\lambda,f)$ corresponds to a path, say $P^\prime$, in $\cG^\prime (\lambda,g)$. 
So, $P^\prime$ is from the vertex $g(0)$ to the vertex $0$,
and also $P^\prime$ has $q$ vertices.

Now, let $x=g(0)$, and let $(x, y)$ be the first edge in the path $P^\prime$. 
Clearly, there are two cases depending on whether $y=g(x^2)$ or $y=g(\lambda x^2)$. 

We first consider the case that $y=g(x^2)$. 
Define the polynomial $h(X) = g(\lambda g(X^2)^2)$.
Then, the path $P^\prime$ is of the form in Figure \ref{fig:perm-P1}, where $x=g(0)$.
So, from Figure \ref{fig:perm-P1}, in this case the number of vertices in $P^\prime$ is at most
$$
2|\cV(h)| + 2 \le q-1,
$$
where the inequality follows from $|\cV(h)|=|\cV (g(X^2)^2)|$ and \eqref{eq:Vg2}.
This contradicts the fact that $P^\prime$ has $q$ vertices.

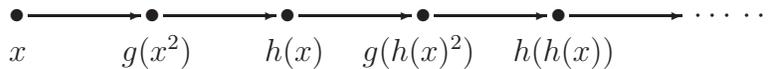
\begin{figure}[!htbp]
\begin{center}
\setlength{\unitlength}{1cm}
\begin{picture}(20,1)
\multiput(1.5,0.5)(1.8,0){5}{$\bullet$}
\multiput(1.75, 0.6)(1.8,0){5}{\vector(1,0){1.5}}
\multiput(10.6,0.5)(0.65,0){2}{$\cdots$}

\put(1.5,0){$x$}
\put(3,0){$g(x^2)$}
\put(4.9,0){$h(x)$}
\put(6.2,0){$g(h(x)^2)$}
\put(8.2,0){$h(h(x))$}
\end{picture}
\end{center}
\caption{The first case of $P^\prime$}
\label{fig:perm-P1}
\end{figure}

Similarly, for the other case that  $y=g(\lambda x^2)$,
we define the polynomial $u(X) = g( g(\lambda X^2)^2)$.
Then, the path $P^\prime$ is of the form in Figure \ref{fig:perm-P2}, where $x=g(0)$.
Thus, from Figure \ref{fig:perm-P2}, in this case the number of vertices in $P^\prime$ is at most
$$
2|\cV(u)| + 2 \le q-1,
$$
where the inequality follows from $|\cV(u)|=|\cV (g(\lambda X^2)^2)|$ and \eqref{eq:Vg2}.
This contradicts the fact that $P^\prime$ has $q$ vertices.

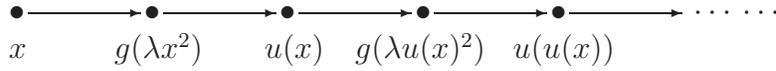
\begin{figure}[!htbp]
\begin{center}
\setlength{\unitlength}{1cm}
\begin{picture}(20,1)
\multiput(1.5,0.5)(1.8,0){5}{$\bullet$}
\multiput(1.75, 0.6)(1.8,0){5}{\vector(1,0){1.5}}
\multiput(10.6,0.5)(0.65,0){2}{$\cdots$}

\put(1.5,0){$x$}
\put(2.9,0){$g(\lambda x^2)$}
\put(4.9,0){$u(x)$}
\put(6.1,0){$g(\lambda u(x)^2)$}
\put(8.2,0){$u(u(x))$}
\end{picture}
\end{center}
\caption{The second case of $P^\prime$}
\label{fig:perm-P2}
\end{figure}

Therefore, there is no such Hamiltonian cycle of Type 1.
\end{proof}

When the polynomial $f$ is of degree one, we can achieve more.

\begin{theorem}  \label{thm:linear-P}
Let $f \in \F_q[X]$ be a polynomial of degree one with non-zero constant term.
Then, any path of Type 1 in the graph $\cG(\lambda,f)$ contains at most  $\lfloor \frac{3}{4}q + \frac{17}{4} \rfloor$  vertices.
\end{theorem}

\begin{proof}
From Proposition~\ref{prop:linear1} below, we can assume that $f(X)=X+a,  a\in \F_q^*$. 
Let $P$ be any path of Type $1$ in the graph $\cG(\lambda,f)$. 
Let $N$ be the number of vertices in $P$.

Following the arguments and the notation in the proof of Theorem~\ref{thm:HT1}, 
in the case here we have $g(X) = X - a$, 
$$
h(X) = g(\lambda g(X^2)^2) = \lambda (X^2-a)^2-a,
$$
and 
$$
u(X) = g( g(\lambda X^2)^2) =  (\lambda X^2-a)^2-a. 
$$
Then, either $N \le 2 |\cV(h)| + 2$ or $N \le 2 |\cV(u)| + 2$.

Note that $\cV(h)=\cV((X^2-a)^2)=\cV(X^4-2aX^2)$.
Then, noticing $a \in \F_q^*$ and using \cite[Theorem 10]{CGM}, we obtain
\begin{equation}  \label{eq:Vh}  
\begin{split}
|\cV(h)| & \le \frac{q-1}{2\gcd(4,q-1)} + \frac{q+1}{2\gcd(4,q+1)} + 1 \\
& = \frac{3}{8}q - \frac{1}{2\gcd(4,q-1)} + \frac{1}{2\gcd(4,q+1)} + 1 \\
& \le \frac{3}{8}q + \frac{9}{8}.
\end{split}
\end{equation}
Similarly, we obtain 
$$
|\cV(u)|   \le \frac{3}{8}q + \frac{9}{8}.
$$

Finally, collecting the above estimates we have
$$
N \le \frac{3}{4}q + \frac{17}{4}.
$$
This completes the proof.
\end{proof}

We remark that the result in Theorem~\ref{thm:linear-P} does not always hold if $f$ has zero constant term. 
For example, the graph $\cG(2,X)$ over $\F_{19}$ has a path of Type 1 having 18 vertices.

\begin{corollary}  \label{cor:linear-HT1}
Assume $q > 17$. Let $f \in \F_q[X]$ be a polynomial of degree one.
Suppose that  the graph $\cG(\lambda,f)$ is connected.
Then,  the graph $\cG(\lambda,f)$ has no Hamiltonian cycle of Type 1.
\end{corollary}

\begin{proof} 
Since $\cG(\lambda,f)$ is connected, we know that $f$ has non-zero constant term 
(otherwise the vertex 0 itself forms a connected component). 
By Theorem~\ref{thm:linear-P}, if $\frac{3}{4}q + \frac{17}{4} < q$,
then the graph $\cG(\lambda,f)$ has no Hamiltonian cycle of Type 1.
Since $q > 17$, this automatically holds.
\end{proof}

In fact, we can obtain similar results for more permutation polynomials over $\F_q$.

\begin{theorem}  \label{thm:perm-HT1}
Let $f \in \F_q[X]$ be a permutation polynomial of the form $Xw(X^2) + a, w \in \F_q[X], a \in \F_q^*$.
Then, the results in Theorem~\ref{thm:linear-P} and Corollary~\ref{cor:linear-HT1}
still hold for the graph $\cG(\lambda,f)$.
\end{theorem}

\begin{proof}
Since $f=Xw(X^2) + a \in \F_q[X]$ is a permutation polynomial over  $\F_q$,
 there  is a permutation polynomial $g \in \F_q[X]$ such that
both $f(g(X))$ and $g(f(X))$ induce the identity map from $\F_q$ to itself.

As in the proof of Theorem~\ref{thm:linear-P}, it suffices to prove
$$
|\cV (g(X^2)^2)| \le \frac{3}{8}q + \frac{9}{8}, \qquad   |\cV(g(\lambda X^2)^2)| \le  \frac{3}{8}q + \frac{9}{8}.
$$

Denote $n=(q+1)/2$. Note that there are exactly $n$ squares in $\F_q$ (including 0).
Let $a_1, \ldots, a_n \in \F_q$ be all the elements such that $a_i + a$ is a square for each $1 \le i \le n$.
Then, let $b_1, \ldots, b_n \in \F_q$ be such that $b_i w(b_i^2) = a_i$ for each $1 \le i \le n$
(here one should note that $Xw(X^2)$ is also a permutation polynomial over  $\F_q$).

Note that if we have $y = g(x^2)$ for some $x,y \in \F_q$, then $x^2 = f(y) = yw(y^2) +a$,
and so $yw(y^2) = a_i$ for some $1 \le i \le n$, and thus $y = b_i$.
So, we have
$$
\cV (g(X^2)) = \{b_1, \ldots, b_n\},
$$
which implies
$$
\cV (g(X^2)^2) = \{b_1^2, \ldots, b_n^2\}.
$$
Since $\{a_1, \ldots, a_n\} = \cV(X^2-a)$ by construction,
as  in \eqref{eq:Vh} we obtain
$$
 |\{a_1^2, \ldots, a_n^2\}| = |\cV((X^2-a)^2)|  \le \frac{3}{8}q + \frac{9}{8}.
$$
Clearly for any $1 \le i,j \le n, i \ne j$, we have that $a_i = - a_j$ if and only if $b_i = -b_j$
(because $Xw(X^2)$ is a permutation polynomial).
Hence, we have
$$
|\{b_1^2, \ldots, b_n^2\}| = |\{a_1^2, \ldots, a_n^2\}|,
$$
which implies
$$
|\cV (g(X^2)^2)| \le \frac{3}{8}q + \frac{9}{8}.
$$

Similarly, we obtain
$$
 |\cV(g(\lambda X^2)^2)| \le  \frac{3}{8}q + \frac{9}{8}.
$$
This in fact completes the proof.
\end{proof}

Note that when $3 \nmid q-1$,  $X^3 +a \in \F_q[X]$ is a permutation polynomial,
so we immediately obtain the following result from Theorem~\ref{thm:perm-HT1}.

\begin{corollary}  \label{cor:cubic-HT1}
Assume that $q > 17$ and $3 \nmid q-1$. Let $f = X^3 + a \in \F_q[X]$.
Suppose that  the graph $\cG(\lambda,f)$ is connected.
Then,  the graph $\cG(\lambda,f)$ has no Hamiltonian cycle of Type 1.
\end{corollary}

In Sections~\ref{sec:linear} and \ref{sec:cubic},
we will make computations about Hamiltonian cycles of Type 2 and Type 3
for the graphs $\cG(\lambda,X+a)$ and $\cG(\lambda,X^3+a)$ respectively.
The computations suggest that these graphs can have many types of Hamiltonian cycles.

\subsection{Binary sequences derived from Hamiltonian cycles}

Traveling through a cycle in $\cG(\lambda,f)$, 
we can get a binary sequence by recording the weights  of the edges (see Definition~\ref{def:weight}) in the cycle.
Especially, we can get a balancing sequence along any Hamiltonian cyle of a connected component.
The word ``balancing" means that the difference between the number of $0$'s and the number of $1$'s in the sequence is at most $1$.

\begin{theorem}
\label{thm:perm-balance}
For each graph $\cG(\lambda,f)$ with permutation polynomial $f$,
along any Hamiltonian cycle of any connected component, we can get a balancing binary sequence.
\end{theorem}

\begin{proof}
Let $H$ be a Hamiltonian cycle of a connected component $C$ in the graph $\cG(\lambda,f)$.
The existence of $H$ has been confirmed by Theorem~\ref{thm:perm-Ha}.
We can assume that $C$ has a vertex not equal to $0$.
As mentioned the above, we can get a binary sequence by going through the cycle $H$ and recording the weights of the edges in $H$.
So, it remains to prove that this sequence is balancing. 

Let $y$ be any non-zero vertex in $C$. Then, $-y$ is also in $C$.
In fact, by Proposition~\ref{prop:perm-inout} they form a rectangle in $C$ having two possible choices with weights; see Figure~\ref{fig:perm-rec}.
Notice that the cycle $H$ travels every vertex in $C$, automatically including the vertices $x, \pm y, z$ in Figure~\ref{fig:perm-rec}.
Then, it is easy to see that the cycle $H$ goes through the edge from $x$ to $y$ with weight $0$ (respectively, $1$)
if and only if it goes through the edge from $z$ to $-y$ with weight $1$ (respectively, $0$);
also, the cycle $H$ goes through the edge from $x$ to $-y$ with weight $0$ (respectively, $1$)
if and only if it goes through the edge from $z$ to $y$ with weight $1$ (respectively, $0$).

\begin{figure}[!htbp]
\begin{center}
\setlength{\unitlength}{1cm}
\begin{picture}(20,2)
\multiput(2.5,1.5)(1.8,0){2}{$\bullet$}
\put(2.75, 1.6){\vector(1,0){1.5}}
\multiput(2.5,-0.3)(1.8,0){2}{$\bullet$}
\put(2.6, 1.45){\vector(0,-1){1.5}}
\put(4.4, -0.05){\vector(0,1){1.5}}
\put(4.25, -0.2){\vector(-1,0){1.5}}
\put(2.2,1.5){$x$}
\put(4.6,1.5){$-y$}
\put(2.2,-0.3){$y$}
\put(4.55,-0.3){$z$}
\put(3.35,1.75){$0$}
\put(2.3,0.6){$0$}
\put(4.5,0.6){$1$}
\put(3.4,-0.05){$1$}

\multiput(7.5,1.5)(1.8,0){2}{$\bullet$}
\put(7.75, 1.6){\vector(1,0){1.5}}
\multiput(7.5,-0.3)(1.8,0){2}{$\bullet$}
\put(7.6, 1.45){\vector(0,-1){1.5}}
\put(9.4, -0.05){\vector(0,1){1.5}}
\put(9.25, -0.2){\vector(-1,0){1.5}}
\put(7.2,1.5){$x$}
\put(9.6,1.5){$-y$}
\put(7.2,-0.3){$y$}
\put(9.55,-0.3){$z$}
\put(8.35,1.75){$1$}
\put(7.3,0.6){$1$}
\put(9.5,0.6){$0$}
\put(8.4,-0.05){$0$}
\end{picture}
\end{center}
\caption{The rectangle related to $\pm y$}
\label{fig:perm-rec}
\end{figure}
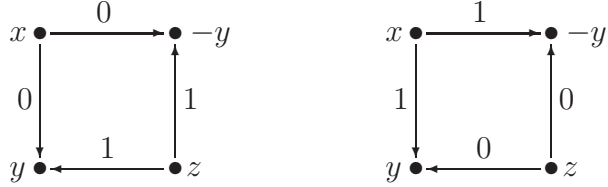

Hence, if the vertex $0$ is not in $C$, then we have the same numbers of $0$'s and $1$'s in the sequence.
Otherwise, if $0$ is a vertex in $C$, then by Definition~\ref{def:weight} the weight of the edge from $f^{-1}(0)$ to $0$ is $0$,
and thus there is exactly one more $0$ than $1$'s in the sequence.
This completes the proof.
\end{proof}

\section{Algorithms}
\label{sec:alg}

In this section, we describe briefly the algorithms we use for counting connected components and searching Hamiltonian cycles
in a graph  $\cG(\lambda,f)$.
We then use them to make computations for the cases when $f=X+a$ and $f=X^3+a$
in Sections~\ref{sec:linear} and \ref{sec:cubic} respectively.

\subsection{Counting connected components}

In order to count connected components in a graph  $\cG(\lambda,f)$, 
we first build the edge table, that is,
the table containing all the edges $(x,y)$ such that $(x,y)$ satisfies either
$y^2 = f(x)$ or $y^2 = \lambda f(x)$. 

Once the edge table is constructed, we perform the standard depth-first search to find the number of connected components in $\cG(\lambda,f)$. 
Note that, as soon as the edge table is built, this is the same code for any polynomial, but here we focus on both linear and cubic cases.

The linear case when $f=X+a$ is straightforward. For the cubic case when $f=X^3+a$, to speed up the process, we precompute the values of $x^3$ 
for all $x \in \F_q$ as well as the values of square roots in $\F_q$.
Therefore for each $a$, we only need to perform $q$ additions to construct 
the edge table.

\subsection{Searching Hamiltonian cycles}

To enumerate all the Hamiltonian cycles, we use the backtracking algorithm.
At the same time it also depends  on the Types. For example, 
to search Hamiltonian cycles of Type 2, we use this algorithm 
 by disregarding any path that contains a trail of length greater than 2. 
This can be done much faster
since their proportion decreases significantly as $q$ increases.

\section{Linear case}
\label{sec:linear}

Since for any $a\ne 0,b\in \F_q$ the polynomial $aX+b$ is a permutation polynomial, all the results in Section~\ref{sec:perm}
automatically hold for the graph $\cG(\lambda,aX+b)$.
Here, we want to investigate these graphs in more detail.

Recall that $q$ is odd, and $\lambda$ is a non-square element in $\F_q$.

\subsection{Isomorphism classes}

It is easy to find some isomorphism classes of the graphs $\cG(\lambda,aX+b)$ over $\F_q$.

\begin{proposition}
\label{prop:linear1}
For any $a\ne 0,b\in \F_q$,  the graph $\cG(\lambda,aX+b)$ is isomorphic to the graph $\cG(\lambda,X+a^{-2}b)$.
\end{proposition}

\begin{proof}
Let $\psi$ be the bijection map from $\F_q$ to itself defined by $\psi(x)=a^{-1}x$.
Automatically, $\psi$ is a bijection between the vertices of $\cG(\lambda,aX+b)$ and the vertices of $\cG(\lambda,X+a^{-2}b)$.
To prove the isomorphism, it suffices to show that there is an edge from $x$ to $y$ in $\cG(\lambda,aX+b)$
if and only if there is an edge from $\psi(x)$ to $\psi(y)$ in $\cG(\lambda,X+a^{-2}b)$.
This can be done by direct computation.
\end{proof}

\begin{proposition}
\label{prop:linear2}
For any $a\in \F_q$,  the graph $\cG(\lambda, X+a)$ is isomorphic to the graph $\cG(\lambda^{-1}, X+\lambda a)$.
\end{proposition}

\begin{proof}
Note that the isomorphism is induced by the bijection map $\psi$ from $\F_q$ to itself defined by $\psi(x)=\lambda x$.
\end{proof}

From Propositions~\ref{prop:linear1} and \ref{prop:linear2}, we know that to investigate the linear case it suffices to
consider the graphs $\cG(\lambda,X+a)$ when $\lambda$ runs over half of the non-square elements of $\F_q$ and $a$ runs over $\F_q$.

We remark that by reversing the directions of the edges in the graph $\cG(\lambda,X+a)$, 
we exactly obtain the equational graph generated by the equation 
$(Y-X^2+a)(Y- \lambda X^2 + a)=0$.
Note that the equational graph generated by the equation $Y-X^2+a = 0$ in fact has been studied extensively; see \cite{KLMMSS,MSSS,VaSha}.

\subsection{Fixed vertices}

In a graph, we say a vertex is a \textit{fixed vertex} if
there is an edge from the vertex to itself. 
For any $a \in \F_q$,
it is easy to see that the graph $\cG(\lambda,X+a)$
has a fixed vertex if and only if either $a+ \frac{1}{4}$ is a square or $\lambda a+ \frac{1}{4}$ is a square.
So, one graph $\cG(\lambda,X+a)$ can have zero, one, two, three or four fixed vertices.

We first recall some classical results for character sums with polynomial arguments, 
which are special cases of \cite[Theorems 5.41 and 5.48]{LN}.

\begin{theorem}
\label{thm:Weil}
Let $\chi$ be the multiplicative quadratic character of $\F_q$, and let $f\in\F_q[X]$ be a polynomial of
positive degree that is not, up to a multiplicative constant, a square of any polynomial. Let $d$ be the number of distinct
roots of $f$ in its splitting field over $\F_q$. Under these conditions, the following inequality holds:
$$
\left| \sum_{x\in\F_q}\chi(f(x))\right|\le (d-1)q^{1/2}.
$$
Moreover, if $f= aX^2 + bX +c$ with $a \ne 0$ and $b^2-4ac \ne 0$, then 
$$
\sum_{x\in\F_q}\chi(f(x)) = - \chi(a). 
$$
\end{theorem}

We now want to count how many graphs $\cG(\lambda,X+a)$ have a fixed vertex.

\begin{proposition}   \label{prop:linear-fix}
Define the set 
$$
S_\lambda = \{a \in \F_q: \, \textrm{$\cG(\lambda,X+a)$ has a fixed vertex} \}. 
$$
Then, we have
$$
|S_\lambda| = \frac{1}{4} \big(3q + 1  + \chi(\lambda - 1) - \chi(1-\lambda) \big), 
$$
where $\chi$ is the multiplicative quadratic character of $\F_q$. In particular, we have 
$|S_\lambda| = \frac{1}{4}(3q + 1)$ if $-1$ is a square in $\F_q$, and otherwise 
$|S_\lambda| = \frac{1}{4}(3q - 1)$ or $\frac{1}{4}(3q + 3)$. 
\end{proposition}

\begin{proof}
We first define the set 
$$
T_\lambda = \{a \in \F_q: \, \textrm{$\cG(\lambda,X+a)$ has no fixed vertex} \}.
$$
Since $S_\lambda = \F_q \setminus T_\lambda$, it is equivalent to show that
$$
  |T_\lambda|  =  \frac{1}{4} \big(q - 1  - \chi(\lambda - 1) + \chi(1-\lambda) \big).
$$

Note that by convention, $\chi(0)=0$. 
For any $a \in \F_q$, we have that $a \in T_\lambda$ if and only if
both $a+ \frac{1}{4}$ and $\lambda a+ \frac{1}{4}$ are not squares, that is,
$$
\chi(a+ \frac{1}{4}) = \chi(\lambda a+ \frac{1}{4}) = -1.
$$
So, we obtain
\begin{align*}
|T_\lambda|  = & \frac{1}{4} \sum_{a \in \F_q} \big(1-\chi(a+ \frac{1}{4})\big) \big(1-\chi(\lambda a+ \frac{1}{4})\big) \\
& - \frac{1}{4} (1-\chi(1-\lambda)) - \frac{1}{4} (1+\chi(\lambda -1)), 
\end{align*}
where the last two terms come from the two cases when $a+ \frac{1}{4} = 0$ or $\lambda a+ \frac{1}{4} = 0$. 
Then, expanding the brackets we further have
\begin{align*}
|T_\lambda| = \frac{q}{4} - \frac{1}{2} - \frac{1}{4} \chi(\lambda - 1) + \frac{1}{4} \chi(1- \lambda) + 
\frac{1}{4} \sum_{a \in \F_q} \chi\big((a+ \frac{1}{4})(\lambda a+ \frac{1}{4})\big),
\end{align*}
where we use the fact $\sum_{a \in \F_q} \chi(a) =0$.
Using Theorem~\ref{thm:Weil} and noticing that $\lambda$ is a non-square element, we have
\begin{align*}
 \sum_{a \in \F_q} \chi\big((a+ \frac{1}{4})(\lambda a+ \frac{1}{4})\big) 
& =  \sum_{a \in \F_q} \chi\big(\lambda a^2 + \frac{1}{4}(\lambda + 1)a + \frac{1}{16}\big) \\
& = - \chi(\lambda) = 1. 
\end{align*}
Hence, we obtain
$$
  |T_\lambda|  =  \frac{1}{4} \big(q - 1 - \chi(\lambda - 1) + \chi(1-\lambda)  \big).
$$
This completes the proof.
\end{proof}

\subsection{Small connected components}
\label{sec:linear-conn}

Here we want to determine small connected components  of the graphs $\cG(\lambda,X+a)$.
This implies some kinds of unconnected graphs.
The later computations suggest that they  almost cover all the unconnected graphs in the linear case.

\begin{proposition}   \label{prop:linear-com2}
For  any $a \in \F_q^*$, the graph $\cG(\lambda,X+a)$ has a connected component with two vertices
if and only if $\lambda \ne -1$ and  $a = 2(\lambda +1)/(\lambda -1)^2$.
In particular, if $\lambda \ne -1$ and  $a = 2(\lambda +1)/(\lambda -1)^2$,
then the vertices $2/(1-\lambda), 2/(\lambda -1)$ form a connected component in $\cG(\lambda,X+a)$.
\end{proposition}

\begin{proof}
First, suppose that the graph $\cG(\lambda,X+a)$ has a connected component, say $C$, with two vertices.
By construction and using Proposition~\ref{prop:perm-inout}, both vertices in $C$ have a loop,  and they form a cycle.
Moreover, if one vertex is $x \in \F_q$, then the other must be $-x \in \F_q$.
Note that $x \ne 0$.
Since the vertices $x$ and $-x$ form a cycle, without loss of generality we can assume that
\begin{equation}  \label{eq:linear-com2}
x^2=x+a, \qquad \lambda x^2 = -x +a.
\end{equation}
If $\lambda = -1$, we have $a=0$, which contradicts with $a \in \F_q^*$. So, we must have $\lambda \ne -1$.
From \eqref{eq:linear-com2}, we deduce that
$$
x = \frac{2}{1-\lambda}, \qquad a = \frac{2(\lambda +1)}{(\lambda -1)^2}.
$$

Conversely, if $\lambda \ne -1$ and  $a = 2(\lambda +1)/(\lambda -1)^2$, then the vertex $x = 2/(1-\lambda)$ satisfies \eqref{eq:linear-com2},
and thus the vertices $x$ and $-x$ form a connected component in the graph $\cG(\lambda,X+a)$.
\end{proof}

We remark that if $-1$ is a non-square element in $\F_q$,
then the graph $\cG(\lambda,X)$ has a connected component with two vertices if and only if $\lambda = -1$
(in fact these two vertices are $1, -1$).

\begin{proposition}   \label{prop:linear-com3}
For any $a \in \F_q$, the graph $\cG(\lambda,X+a)$ has a connected component with three vertices
if and only if either $\lambda = 2, a = 1$, or $\lambda = 1/2, a = 2$.
In particular, if $2$ is a non-square element in $\F_q$,
then the vertices $0, 1$ and $-1$ form a connected component in $\cG(2,X+1)$,
and the vertices $0, 2$ and $-2$ form a connected component in $\cG(1/2,X+2)$.
\end{proposition}

\begin{proof}
It is easy to check that if $\lambda = 2$ and  $a = 1$ (by the assumption on $\lambda$, $2$ is a non-square element in $\F_q$),
the vertices $0, 1$ and $-1$ form a connected component in $\cG(2,X+1)$.
Moreover, if $\lambda = 1/2$ and  $a = 2$ ($2$ is a non-square element in $\F_q$),
the vertices $0, 2$ and $-2$ form a connected component in $\cG(1/2,X+2)$.
This shows the sufficiency. It remains to show the necessity.

Now, suppose that the graph $\cG(\lambda,X+a)$ has a connected component, say $C$, with three vertices.
Note that by construction if $x \in \F_q$ is a vertex in $C$, then so is $-x$.
So, the vertex $0$ must be in $C$.

Since there is an edge from $-a$ to $0$,  by construction we have that the three vertices of $C$ are $0, a$ and $-a$
(so $a \ne 0$), and there are edges from $0$ to $a$ and $-a$.
Moreover, there is an edge from $a$ to $-a$.
These give either (noticing $a \ne 0$)
$$
a^2 = a, \qquad  \lambda a^2 = a +a,
$$
or
$$
\lambda a^2 = a, \qquad   a^2 = a +a.
$$
Hence, we deduce that either $\lambda = 2, a = 1$, or $\lambda = 1/2, a = 2$.
\end{proof}

We remark that by Proposition~\ref{prop:linear2}, the graph $\cG(2,X+1)$ is isomorphic to the graph $\cG(1/2,X+2)$.
By Proposition~\ref{prop:linear-com3} we also know that if $2$ is a square element in $\F_q$, then connected components with three vertices
can not occur in the graphs $\cG(\lambda,X+a)$ over $\F_q$ (note that $\lambda$ is set to be a non-square element in $\F_q$ throughout the paper).

However, there are few connected components having four vertices.

\begin{proposition}   \label{prop:linear-com4}
For any $a \in \F_q^*$, the graph $\cG(\lambda,X+a)$ has no connected component with four vertices.
\end{proposition}

\begin{proof}
By contradiction, we assume that the graph $\cG(\lambda,X+a)$
has a connected component, say $C$, having four vertices.
By Proposition~\ref{prop:perm-inout}, it is easy to see that $C$ has only two possible cases; see Figure~\ref{fig:linear-4}.

\begin{figure}[!htbp]
\begin{center}
\setlength{\unitlength}{1cm}
\begin{picture}(20,2)
\multiput(2.5,1.5)(1.8,0){2}{$\bullet$}

\put(2.75, 1.7){\vector(1,0){1.5}}
\put(4.25, 1.5){\vector(-1,0){1.5}}

\multiput(2.5,-0.3)(1.8,0){2}{$\bullet$}

\put(2.5, 1.45){\vector(0,-1){1.5}}
\put(2.7, -0.05){\vector(0,1){1.5}}

\put(4.5, -0.05){\vector(0,1){1.5}}
\put(4.3, 1.45){\vector(0,-1){1.5}}

\put(4.25, -0.3){\vector(-1,0){1.5}}
\put(2.75, -0.1){\vector(1,0){1.5}}

\put(2.2,1.5){$x$}
\put(4.6,1.5){$-y$}
\put(2.2,-0.3){$y$}
\put(4.55,-0.3){$-x$}

\multiput(7.5,1.5)(1.8,0){2}{$\bullet$}
\multiput(7.5,-0.3)(1.8,0){2}{$\bullet$}

\put(7.6, 1.45){\vector(0,-1){1.5}}
\put(9.4, 1.45){\vector(0,-1){1.5}}
\put(9.25, -0.3){\vector(-1,0){1.5}}
\put(9.25, 0){\vector(-1,1){1.5}}
\put(7.75, -0.1){\vector(1,0){1.5}}
\put(7.75, 0){\vector(1,1){1.5}}

\put(7.65,1.7){\oval(0.3,0.8)[t]}
\put(7.8, 1.7){\vector(0,-1){0}}
\put(9.35,1.7){\oval(0.3,0.8)[t]}
\put(9.2, 1.7){\vector(0,-1){0}}

\put(7.2,1.5){$x$}
\put(9.6,1.5){$y$}
\put(6.8,-0.3){$-x$}
\put(9.55,-0.3){$-y$}
\end{picture}
\end{center}
\caption{Two cases of a component with four vertices}
\label{fig:linear-4}
\end{figure}
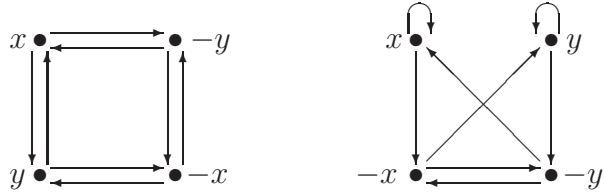

We now consider the first case that there is no fixed vertex.
Without loss of generality, we can assume that
$$
y^2 = x + a, \qquad \lambda y^2 = -x+a.
$$
Moreover, either
$$
x^2 = y + a, \qquad \lambda x^2 = -y+a,
$$
or
$$
x^2 = -y + a, \qquad \lambda x^2 = y+a.
$$
Then, noticing $x \ne \pm y$, we obtain $\lambda = -1, a = 0, x^3 =1$ and $x \ne 1$.
This contradicts with $a \ne 0$.
So, the first case cannot happen.

For the second case in Figure~\ref{fig:linear-4},
 noticing $x \ne \pm y$, we only need to consider the following four subcases:
\begin{itemize}
\item[(1)]  $x^2 = x + a, \qquad y^2 = -x+a, \qquad \lambda y^2 = y + a, \qquad \lambda x^2 = -y + a$;

\item[(2)]  $x^2 = x + a, \qquad \lambda y^2 = -x+a, \qquad  y^2 = y + a, \qquad \lambda x^2 = -y + a$;

\item[(3)]  $\lambda x^2 = x + a, \qquad  y^2 = -x+a, \qquad \lambda y^2 = y + a, \qquad  x^2 = -y + a$;

\item[(4)]  $\lambda x^2 = x + a, \qquad \lambda y^2 = -x+a, \qquad  y^2 = y + a, \qquad  x^2 = -y + a$.

\end{itemize}
By direct calculations,
from Cases (2) and (3) we obtain $\lambda = 1$, which contradicts with the assumption that $\lambda$ is non-square;
and Case (1) gives $\lambda^2 = -1, a = 0, x=1, y= -\lambda$; Case (4) gives $\lambda^2 = -1, a = 0, x = -\lambda, y =1$.
Hence,  the second case also cannot happen (due to  $\lambda$ non-square and $a \ne 0$).

Therefore, there is no connected component with four vertices in the graph $\cG(\lambda,X+a)$ with $a \in \F_q^*$.
\end{proof}

When proving Proposition~\ref{prop:linear-com4}, we in fact have obtained the following result
about the graph $\cG(\lambda,X)$.

\begin{proposition}
If $-1$ is a non-square element in $\F_q$,
then the graph $\cG(\lambda,X)$ has a connected component with four vertices if and only if $3 \mid q-1$ and  $\lambda = -1$
(in fact these four vertices are $x, -x, 1/x, -1/x$, where $x^3=1$ and $x \ne 1$, corresponding to the first case in Figure~\ref{fig:linear-4}).
Otherwise, if $-1$ is a square in $\F_q$,
then the graph $\cG(\lambda,X)$ has a connected component with four vertices if and only if $4 \mid q-1$ and $\lambda^2 = -1$
(in fact these four vertices are $1, -1, \lambda, -\lambda$, corresponding to the second case in Figure~\ref{fig:linear-4}).
\end{proposition}

We remark that connected components with five vertices exist.
For example, in the graph $\cG(3,X+2)$ over $\F_7$ the vertices $0,2,3,4,5$ form a connected component,
and the graph $\cG(10,X+12)$ over $\F_{17}$ has a connected component with 5 vertices (that is, $0, 4, 5, 12, 13$).
Moreover, the graph $\cG(5,X+8)$ over $\F_{17}$ has a connected component with 6 vertices (that is, $1, 3, 7, 10, 14, 16$),
and the graph $\cG(3,X+13)$ over $\F_{31}$ has a connected component with 6 vertices (that is, $3, 4, 14, 17, 27, 28$).

\subsection{Computations concerning connectedness}
\label{sec:linear-com}

Recall that $p$ is an odd prime.
Here, we want to make some computations for the graphs $\cG(\lambda,X+a)$ over $\F_p$ concerning its connectedness.
From \cite[Section 2]{MSSS}, the numerical results suggest that almost all the functional graphs generated by polynomials $f(X)=X^2 +a$
($a$ runs over $\F_p^*$) are weakly unconnected.
However, our computations suggest that almost all the graphs
$\cG(\lambda,X+a)$ are connected (in fact, strongly connected by Proposition~\ref{prop:perm-conn}).

By Proposition~\ref{prop:linear2} we do not need to consider all the non-square elements $\lambda$.
We first identify the elements of $\F_p$ as the set $\{0,1, \ldots, p-1\}$, and then
 define $\N_p$ to be the subset of non-square elements in $\F_p$ such that for any non-square element $\lambda$,
only the smaller one of  $\lambda$ and $\lambda^{-1}$ is contained in $\N_p$.
Clearly, for the size of $\N_p$, we have
\begin{equation}
\label{eq:Np}
|\N_p| =
\left\{\begin{array}{ll}
 (p-1)/4 & \text{if $p \equiv 1 \pmod{4}$,}\\
 (p+1)/4 & \text{if $p \equiv 3 \pmod{4}$,}\\
\end{array}\right.
\end{equation}
where we use the fact that $-1$ is non-square modulo $p$ if and only if $p \equiv 3 \pmod{4}$.
So, here we make computations for the graphs $\cG(\lambda,X+a)$ when $\lambda$ runs over $\N_p$ and $a$ runs over $\F_p^*$.
By Propositions \ref{prop:linear1} and \ref{prop:linear2}, this indeed includes all the linear cases over $\F_p$ except the case $\cG(\lambda,X)$.

Let $C_1(p)$ (respectively, $U_1(p)$) be the number of connected (respectively, unconnected) graphs
among all the graphs $\cG(\lambda,X+a)$ when $\lambda$ runs over $\N_p$ and $a$ runs over $\F_p^*$,
and let $R_1(p)$ be the ratio of connected graphs, that is
$$
R_1(p) = \frac{C_1(p)}{C_1(p) + U_1(p)}.
$$
In Table~\ref{tab:linear-conn1} we record the first five decimal digits of the ratio $R_1(p)$,
and we will do the same for all the other ratios and the average numbers.
Further, let $M_1(p)$ be the maximum number of  connected components  of an unconnected graph counted in $U_1(p)$.

Define
\begin{equation}
\label{eq:Lp}
L(p) =
\left\{\begin{array}{ll}
 (p-1)/4 & \text{if $p \equiv 1 \pmod{8}$,}\\
 (p+1)/4 & \text{if $p \equiv 3 \pmod{8}$,}\\
 (p+3)/4 & \text{if $p \equiv 5 \pmod{8}$,}\\
 (p-3)/4 & \text{if $p \equiv 7 \pmod{8}$.}
\end{array}\right.
\end{equation}
It is well-known that $2$ is a non-square element modulo $p$ if and only if $p \equiv 3,5 \pmod{8}$.
In view of the construction of $\N_p$, we always have $1/2 \pmod{p} \not\in \N_p$.
So,  the number of the unconnected graphs counted in $U_1(p)$ and  described as in Propositions~\ref{prop:linear-com2} and \ref{prop:linear-com3}
is equal to $L(p)$ when $p > 5$ (using also \eqref{eq:Np}).
Hence, we have
\begin{equation}
U_1(p) \ge L(p),  \quad p > 5.
\end{equation}
Notice that when $q=5$, the graphs in Propositions~\ref{prop:linear-com2} and \ref{prop:linear-com3} coincide.

From Table~\ref{tab:linear-conn1}, we can see that almost all the graphs $\cG(\lambda,X+a)$ are connected,
and $U_1(p)$ is quite close to $L(p)$.
Moreover, the data suggest that each unconnected graph $\cG(\lambda,X+a)$ with $a \ne 0$ has exactly two connected components,
and its small connected component usually has exactly two vertices.

\begin{table}[!htbp]
\centering
\begin{tabular}{|c|c|c|c|c|c|}
\hline
$p$ & $C_1(p)$ & $U_1(p)$ & $L(p)$ & $M_1(p)$ & $R_1(p)$ \\ \hline

31 & 232 & 8   &  7 & 2  & 0.96666  \\ \hline

107 & 2835 & 27  & 27  & 2 & 0.99056 \\ \hline

523 & 68251 & 131  & 131 & 2 & 0.99808  \\ \hline

1009 & 253764 & 252  & 252 & 2  & 0.99900 \\ \hline

1511 & 570403  & 377  & 377 & 2 & 0.99933 \\ \hline

2029 & 1027688 & 508  & 508 & 2 & 0.99950 \\ \hline

2521 & 1586970 & 630  & 630 & 2 & 0.99960 \\ \hline

3037 & 2303564 & 760  & 760 & 2 & 0.99967 \\ \hline

4049 & 4095564  &  1012 & 1012 & 2 & 0.99975 \\ \hline

5003 & 6256251 & 1251  & 1251  & 2 & 0.99980 \\ \hline
\end{tabular}
\bigskip
\caption{Counting connected graphs in linear case}
\label{tab:linear-conn1}
\end{table}

\begin{conjecture}  \label{conj:linear-conn0}
Almost all the graphs $\cG(\lambda,X+a)$ over $\F_p$ are connected when $p$ goes to infinity.
\end{conjecture}

We in fact make computations for all primes $p \le 3041$ and find that $U_1(p) = L(p)$ for any $31< p \le 3041$,
and for any $5 \le p \le 3041$ each unconnected graph $\cG(\lambda,X+a)$ with $a \ne 0$ over $\F_p$
has exactly two connected components.
Hence, we make the following conjectures.

\begin{conjecture}  \label{conj:linear-conn1}
For any prime $p>31$, $U_1(p) = L(p)$.
\end{conjecture}

\begin{conjecture}  \label{conj:linear-conn2}
For any prime $p \ge 5$, each unconnected graph $\cG(\lambda,X+a)$ with $a \ne 0$ over $\F_p$
has exactly two connected components.
\end{conjecture}

Conjecture~\ref{conj:linear-conn1} suggests that when $p>31$,  if  a graph $\cG(\lambda,X+a)$ over $\F_p$ with $a \ne 0$
does not belong to the cases described in Propositions~\ref{prop:linear-com2} and \ref{prop:linear-com3},
then  it is a connected graph.

\subsection{Hamiltonian cycles}

For each graph $\cG(\lambda,X+a)$, Theorem~\ref{thm:perm-Ha} has confirmed that all its connected components have a Hamiltonian cycle.
Here, we only make computations on Hamiltonian cycles of $\cG(\lambda,X+a)$ over $\F_p$ when it is a connected graph.
Due to the complexity, we only can test small primes $p$.

Let $H_{11}(p)$ (respectively, $H_{12}(p)$) be the minimal (respectively, maximal) number of Hamiltonian cycles in a connected graph of the form $\cG(\lambda,X+a)$
($\lambda$ runs over $\N_p$ and $a$ runs over $\F_p^*$).
Then, let $H_1(p)$ be the average number of Hamiltonian cycles in these connected graphs.

By Corollary~\ref{cor:linear-HT1}, a connected graph $\cG(\lambda,X+a)$ over $\F_p$ has no Hamiltonian cycle of Type $1$ when $p> 17$.
Let $R_{12}(p)$ (respectively, $R_{13}(p)$) be the ratio of such connected graphs having Hamiltonian cycles of Type $2$ (respectively, Type $3$)
over all the connected graphs (when $\lambda$ runs over $\N_p$ and $a$ runs over $\F_p^*$).
Let $A_{12}(p)$ (respectively, $A_{13}(p)$) be the average number of  Hamiltonian cycles of Type $2$ (respectively, Type $3$)
over all such connected graphs having Hamiltonian cycles of Type $2$ (respectively, Type $3$).

From Table~\ref{tab:linear-hamilton}, we can see that there are many Hamiltonian cycles in a connected graph $\cG(\lambda,X+a)$ over $\F_p$,
whose amount grows rapidly with respect to $p$.

\begin{table}[!htbp]
\centering
\begin{tabular}{|c|c|c|c|c|}
\hline
$p$ & $H_{11}(p)$ & $H_{12}(p)$ & $H_1(p)$ \\ \hline

17 & 1 & 72  &  18.31034   \\ \hline

19 & 4 & 148 &   34.03529  \\ \hline

23 & 5  &  423 &  93.70078  \\ \hline

29 & 34 & 2840   &  666.13829  \\ \hline

31 & 30 & 5410  &  1206.08620  \\ \hline

37  & 448 & 45546  &  7906.61783  \\ \hline

41 & 1223 & 175428  &  28473.26666  \\ \hline

43 & 2222 & 255558  & 53999.07760  \\ \hline

47  & 6576 & 1273729  &  195723.05914  \\ \hline

53 & 63363  & 6795031  &  1297781.68277  \\ \hline
\end{tabular}
\bigskip
\caption{Counting Hamiltonian cycles in linear case}
\label{tab:linear-hamilton}
\end{table}

Tables~\ref{tab:linear-HT2} and \ref{tab:linear-HT3} suggest that although many connected graphs have Hamiltonian cycles of Types 2 and 3,
 these two types of Hamiltonian cycles occupy a small proportion when $p$ is large.
 This means that each connected graph $\cG(\lambda,X+a)$ over $\F_p$ is likely to have many types of Hamiltonian cycles.

\begin{table}[!htbp]
\centering
\begin{tabular}{|c|c|c|c|}
\hline
$p$  & $A_{12}(p)$ & $A_{12}(p)/H_1(p)$ & $R_{12}(p)$ \\ \hline

17   &  2.40909 & 0.13156 & 0.75862  \\ \hline

19  & 3.30158 & 0.09700  & 0.74117 \\ \hline

23  & 3.79452  & 0.04049 & 0.57480 \\ \hline

29  & 8.30434  & 0.01246 & 0.61170 \\ \hline

31  & 8.25954  & 0.00684 & 0.56465 \\ \hline

37  & 17.72580  & 0.00224 & 0.59235 \\ \hline

41  & 39.12643  & 0.00137 & 0.44615 \\ \hline

43  & 32.31088  & 0.00059 & 0.42793 \\ \hline

47  &  60.58173 & 0.00030 & 0.38447 \\ \hline

53   & 107.67010  & 0.00008  & 0.43957\\ \hline
\end{tabular}
\bigskip
\caption{Hamiltonian cycles of Type 2 in linear case}
\label{tab:linear-HT2}
\end{table}

\begin{table}[!htbp]
\centering
\begin{tabular}{|c|c|c|c|c|c|c|}
\hline
$p$  & $A_{13}(p)$ & $A_{13}(p)/H_1(p)$ & $R_{13}(p)$ \\ \hline

17   & 7.42857 & 0.40570 &  0.96551 \\ \hline

19   & 11.38095 & 0.33438 & 0.98823  \\ \hline

23  & 25.72580 & 0.27455 & 0.97637 \\ \hline

29   & 125.45161 & 0.18832 & 0.98936 \\ \hline

31    & 190.37117 & 0.15784 & 0.98706 \\ \hline

37  & 754.65594 & 0.09544 & 0.99044 \\ \hline

41  & 2036.97927  & 0.07154 & 0.98974 \\ \hline

43  & 3296.68456 & 0.06105 & 0.99113 \\ \hline

47 & 7935.52918 & 0.04054 & 0.95009 \\ \hline

53  & 35505.00762  & 0.02735 & 0.99093  \\ \hline
\end{tabular}
\bigskip
\caption{Hamiltonian cycles of Type 3 in linear case}
\label{tab:linear-HT3}
\end{table}

\section{Quadratic case}
\label{sec:quad}

For the quadratic case we only establish some general properties without making computations.
These suggest that this case might be not attractive.
The main reason is that quadratic polynomials are not permutation polynomials.

Recall that $q$ is odd, and $\lambda$ is a non-square element in $\F_q$.

\begin{proposition}
\label{prop:quad-iso1}
For any $a\ne 0,b\in \F_q$,  the graph $\cG(\lambda,X^2+aX+b)$ is isomorphic to the graph $\cG(\lambda,X^2+X+a^{-2}b)$.
\end{proposition}

\begin{proof}
Note that the isomorphism is induced by the bijection map $\psi$ from $\F_q$ to itself defined by $\psi(x)=a^{-1} x$.
\end{proof}

\begin{proposition}
\label{prop:quad-iso2}
For any $a \ne 0, b \ne 0$,  if $b/a$ is a square element in $\F_q$, then the graph $\cG(\lambda,X^2+a)$ is isomorphic to the graph $\cG(\lambda,X^2+b)$.
\end{proposition}

\begin{proof}
Write $b/a = c^2, c \in \F_q$.
Note that the isomorphism is induced by the bijection map $\psi$ from $\F_q$ to itself defined by $\psi(x)=cx$.
\end{proof}

From Proposition~\ref{prop:quad-iso2}, we know that for a fixed $\lambda$,
there are at most two graphs up to isomorphism among the graphs $\cG(\lambda,X^2+a)$ when $a$ runs over $\F_q^*$.

We remark that in the graph $\cG(\lambda,X^2+X+a)$, if $-a + 1/4$ is not a square element, then the in-degree of the vertex $0$ is zero.
Similarly, in the graph $\cG(\lambda,X^2+a)$, if $-a$ is not a square element in  $\F_q$, then the in-degree of the vertex $0$ is zero.
In fact, there could be many vertices with zero in-degree.

\begin{proposition}
\label{prop:quad-indeg1}
For each graph $\cG(\lambda, X^2+X+a)$ over $\F_q$ with $a \ne 1/4$,
there are at least $\lfloor \frac{1}{4} \big( q - 3\sqrt{q} \big) -1 \rfloor$ vertices having zero in-degree.
\end{proposition}

\begin{proof}
Denote by $Z(\lambda, a)$ the number of vertices having zero in-degree in the graph $\cG(\lambda, X^2+X+a)$.
Let $\chi$ be the multiplicative quadratic character of $\F_q$. By convention, $\chi(0)=0$.
If a vertex $y \in \F_q$ has zero in-degree, then both $y^2 -a +1/4$ and $\lambda y^2 -a +1/4$ are non-square elements in $\F_q$,
that is,
$$
\chi(y^2 -a +1/4) =  \chi(\lambda y^2 -a +1/4) = -1.
$$
So,  we have
$$
Z(\lambda, a) \ge \sum_{y \in \F_q} \frac{1-\chi(y^2 -a +1/4)}{2} \cdot \frac{1-\chi(\lambda y^2 -a +1/4)}{2} - 1,
$$
where the term ``$- 1$" comes from one of the two cases when either $y^2 -a +1/4=0$ or $\lambda y^2 -a +1/4=0$
(since $a \ne 1/4$ and $\lambda$ is a non-square element in $\F_q$, these two cases can not both happen).

Then, using Theorem~\ref{thm:Weil} and noticing $a \ne 1/4$, we deduce that
\begin{align*}
Z(\lambda, a) & \ge \sum_{y \in \F_q} \frac{1-\chi(y^2 -a +1/4)}{2} \cdot \frac{1-\chi(\lambda y^2 -a +1/4)}{2} -1 \\
& = \frac{1}{4} \Big( q - \sum_{y \in \F_q} \chi(y^2 -a +1/4) -  \sum_{y \in \F_q} \chi(\lambda y^2 -a +1/4) \\
& \qquad\qquad + \sum_{y \in \F_q} \chi((y^2 -a +1/4)(\lambda y^2 -a +1/4))\Big) - 1 \\
& \ge \frac{1}{4} \Big( q + \chi(1) + \chi(\lambda) - 3\sqrt{q} \Big) -1 \\
& =   \frac{1}{4} \big( q - 3\sqrt{q} \big) -1.
\end{align*}
This in fact completes the proof.
\end{proof}

Similarly, we obtain:

\begin{proposition}
\label{prop:quad-indeg2}
For each graph $\cG(\lambda, X^2+a)$ over $\F_q$ with $a \ne 0$,
there are at least $\lfloor \frac{1}{4} \big( q - 3\sqrt{q} \big) -1 \rfloor$ vertices having zero in-degree.
\end{proposition}

When $q \ge 23$, we have $ \frac{1}{4} \big( q - 3\sqrt{q} \big) -1  \ge 1$.
 Propositions~\ref{prop:quad-indeg1} and \ref{prop:quad-indeg2} imply that strongly connected graphs are rare in the quadratic case.
We also can say something about  Hamiltionian cycles.

\begin{proposition}
\label{prop:quad-H1}
For each graph $\cG(\lambda,X^2+X+a)$, if $a \ne  1/4$,
then  its connected component containing the vertex $0$ does not have a Hamiltonian cycle.
\end{proposition}

\begin{proof}
Notice that in the graph $\cG(\lambda,X^2+X+a)$, if $a \ne  1/4$, then the vertex $0$ either has zero in-degree or has in-degree $2$.
If the vertex $0$ has zero in-degree, then the component automatically has no Hamiltonian cycle.
If the vertex $0$ has in-degree $2$, let $x_1$ and $x_2$ be the two predecessors of $0$.
Note that the vertex $0$ is the only successor of $x_1$ and $x_2$.
So, there is no cycle going through both $x_1$ and $x_2$.
This completes the proof.
\end{proof}

Similarly, we have:

\begin{proposition}
\label{prop:quad-H2}
For each graph $\cG(\lambda,X^2+a)$, if $a \ne 0$, then  its connected component containing the vertex $0$ does not have a Hamiltonian cycle.
\end{proposition}

We remark that in Proposition~\ref{prop:quad-H1}, if $a= 1/4$, then the graph $\cG(\lambda,X^2+X+1/4)$ is in fact generated by the equation 
$(Y - X - 1/2)(Y + X + 1/2) = 0$.

\section{Cubic case}
\label{sec:cubic}

For each polynomial $X^3 + aX +b$ over $\F_q$, if $4a^3+27b^2 \ne 0$,
then the equation $Y^2 =X^3 + aX + b$ defines an elliptic curve over $\F_q$.
In this section, we consider the graphs $\cG(\lambda,X^3+ aX + b)$ over $\F_q$.

Recall that $q$ is odd, and $\lambda$ is a non-square element in $\F_q$.

\subsection{Basic properties}

As before, we can easily find some isomorphism classes among the graphs $\cG(\lambda,X^3+aX+b)$
when $\lambda$ runs over the non-square elements and $a,b$ run over $\F_q$.

\begin{proposition}
\label{prop:cubic1}
For any $a,b\in \F_q$,  the graph $\cG(\lambda,X^3+aX+b)$ is isomorphic to
the graph $\cG(\lambda^{-1},X^3+  \lambda^{-2}aX + \lambda^{-3}b)$.
\end{proposition}

\begin{proof}
Note that the isomorphism is induced by the bijection map $\psi$ from $\F_q$ to itself defined by $\psi(x)=\lambda^{-1} x$.
\end{proof}

When the characteristic of $\F_q$ (that is, $p$) is greater than 3,
 it is not hard to show that a polynomial $X^3+aX+b$ is a permutation polynomial over $\F_q$
 if and only if $3 \nmid q-1$ and $a=0$; see \cite[Theorem 2.2]{MW}.

So, in the sequel,  we only consider the graph $\cG(\lambda,X^3+a)$.
Note that if $3 \nmid q-1$, then each polynomial $X^3+a$ is a permutation polynomial over $\F_q$,
and thus all the results in Section~\ref{sec:perm} automatically hold for the graph $\cG(\lambda,X^3+a)$.

We know that each vertex in the graph $\cG(\lambda, X^3+a)$ has positive out-degree.
However, this is not always true for the in-degree.
If $3 \nmid q-1$, by Proposition~\ref{prop:perm-inout}
each vertex in the graph $\cG(\lambda, X^3+a)$ has positive in-degree.
However, when $3 \mid q-1$, there are many vertices having zero in-degree,
and more precisely we can get a similar result as in Proposition~\ref{prop:quad-indeg2}
by using a different approach and in a stronger form.

\begin{proposition}
\label{prop:cubic-indeg}
If $3 \mid q-1$, then for each graph $\cG(\lambda, X^3+a)$ over $\F_q$, there are at least $N$ vertices having zero in-degree, where
\begin{equation*}
N =
\left\{\begin{array}{ll}
 (q-1)/3 & \text{if $-a$ is a cubic element in $\F_q$,}\\
 (q-7)/3  & \text{otherwise.}\\
\end{array}\right.
\end{equation*}
In particular, there exists at least one vertex in $\cG(\lambda, X^3+a)$ having zero in-degree.
\end{proposition}

\begin{proof}
Let  $Q = \{x^3: \, x\in \F_q \}$.
Since $3 \mid q-1$, we have
$$
|Q| = \frac{q-1}{3} + 1.
$$
Let $R = \F_q \setminus Q$, and let $S$ be the set of equivalence classes of $\F_q$ modulo $\pm 1$.
Clearly, we have
\begin{equation}
\label{eq:RS}
|R| = q -  \left(\frac{q-1}{3} + 1 \right) = \frac{2(q-1)}{3},  \qquad |S| = \frac{q+1}{2}.
\end{equation}

Define the map $\varphi$ from $R$ to $S$ by $\varphi(x) = \{\pm y\}$ 
if either $y^2=x+a$ or $\lambda y^2=x +a$.
If $x_1,x_2 \in R$ with $x_1 \ne x_2$ such that $\varphi(x_1)=\varphi(x_2)=\{\pm y_0\}$ for some $y_0 \in \F_q$, then
either $y_0^2=x_1+a, \lambda y_0^2 = x_2+a$, or $\lambda y_0^2=x_1+a,  y_0^2 = x_2+a$.
This implies that there is no $x_3 \in R$ with $x_3 \ne x_1$ and $x_3 \ne x_2$ such that $\varphi(x_3)=\{\pm y_0\}$.
By the definition of $\varphi$,  both $y_0^2-a$ and $\lambda y_0^2-a$ are not in $Q$,
and thus the vertices $\pm y_0$ have zero in-degree in the graph $\cG(\lambda, X^3+a)$.

Now, define the set
$$
T = \{x\in R: \, \exists \ x^\prime \in R, x^\prime \ne x, \varphi(x^\prime)=\varphi(x) \}.
$$
By the above discussion, we know that $|T|$ is even, and the number of vertices having zero in-degree is at least $|T|$.
So, it suffices to get a lower bound for $|T|$.

Considering the size of $\varphi(R)$ and noticing $|\varphi(T)| =|T|/2$, we have
\begin{equation}
\label{eq:RST}
|\varphi(R)| = |R| - |T|/2 \le |S|,
\end{equation}
which, together with \eqref{eq:RS}, implies that
$$
|T| \ge \frac{q-7}{3}.
$$

Moreover, if $-a \in Q$ (that is, $-a$ is a cubic element in $\F_q$), then there is no  $x \in R$ such that $\varphi(x) = \{0\}$.
So, the inequality in \eqref{eq:RST} becomes
$$
|\varphi(R)| = |R| - |T|/2 \le |S| - 1,
$$
which gives
$$
|T| \ge \frac{q-1}{3}.
$$
This completes the proof for the choice of $N$.

For the final claim, by the choice of $N$, we only need to consider the case $q=7$.
By direction computation, if $q=7$, indeed there exists at least one vertex in each graph $\cG(\lambda, X^3+a)$ having zero in-degree.
\end{proof}

As in Proposition~\ref{prop:quad-H2}, we have:

\begin{proposition}
If $3 \mid q-1$, then for each graph $\cG(\lambda,X^3+a)$ with $a \ne 0$,  its connected component containing the vertex $0$ does not have a Hamiltonian cycle.
\end{proposition}

\subsection{Small connected components}
\label{sec:cubic-conn}

As in Section~\ref{sec:linear-conn}, we determine small connected components for the graphs $\cG(\lambda,X^3+a)$
over $\F_q$ when $3 \nmid q-1$.

\begin{proposition}   \label{prop:cubic-com2}
Assume $3 \nmid q-1$.
For  any $a \in \F_q^*$, the graph $\cG(\lambda,X^3+a)$ has a connected component with two vertices
if and only if $\lambda \ne -1$ and  $a = (\lambda +1)(\lambda -1)^2/8$.
In particular, if $\lambda \ne -1$ and  $a = (\lambda +1)(\lambda -1)^2/8$,
then the vertices $(1-\lambda)/2, (\lambda -1)/2$ form a connected component in $\cG(\lambda,X^3+a)$.
\end{proposition}

\begin{proof}
As before, if the graph $\cG(\lambda,X^3+a)$ has a connected component with two vertices,
then these two vertices are $\pm x$ for some $x \in \F_q^*$; and so without loss of generality, we can assume
$$
x^2=x^3 +a, \qquad \lambda x^2 = -x^3 +a,
$$
which gives
$$
\lambda \ne -1, \qquad a = (\lambda +1)(\lambda -1)^2/8, \qquad x = (1-\lambda)/2.
$$
The rest is straightforward.
\end{proof}

We remark that if $-1$ is a non-square element in $\F_q$,
then the graph $\cG(\lambda,X^3)$ has a connected component with two vertices if and only if $\lambda = -1$
(in fact these two vertices are $1, -1$).

\begin{proposition}   \label{prop:cubic-com3}
Assume $3 \nmid q-1$.
For any $a \in \F_q$, the graph $\cG(\lambda,X^3+a)$ has a connected component with three vertices
if and only if either $\lambda = 2, a = 1$, or $\lambda = 1/2, a = 1/8$.
In particular, if $2$ is a non-square element in $\F_q$,
then the vertices $0, 1$ and $-1$ form a connected component in $\cG(2,X^3+1)$,
and the vertices $0, 1/2$ and $-1/2$ form a connected component in $\cG(1/2,X^3+1/8)$.
\end{proposition}

\begin{proof}
We only prove the necessity.
As before, if the graph $\cG(\lambda,X^3+a)$ has a connected component with three vertices,
then these three vertices are $0, \pm b$, where $b^3=a$.
So, there must be an edge from $0$ to $b$ and an edge from $b$ to $-b$ (by Proposition~\ref{prop:perm-inout}).
This gives either (noticing $b \ne 0$)
$$
b^2 = a, \qquad  \lambda b^2 = b^3 +a,
$$
or
$$
\lambda b^2 = a, \qquad   b^2 = b^3 +a.
$$
Hence, we obtain either $\lambda = 2, a = 1, b=1$, or $\lambda = 1/2, a = 1/8, b=1/2$.
\end{proof}

We remark that by Proposition~\ref{prop:cubic1}, the graph $\cG(2,X^3+1)$ is isomorphic to the graph $\cG(1/2,X^3+1/8)$.
By Proposition~\ref{prop:cubic-com3} we also know that if $2$ is a square in $\F_q$, then connected components with three vertices
can not occur in the graphs $\cG(\lambda,X^3+a)$ over $\F_q$ (note that $\lambda$ is set to be a non-square element in $\F_q$ throughout the paper).

\begin{proposition}   \label{prop:cubic-com4}
Assume $3 \nmid q-1$.
For any $a \in \F_q^*$, the graph $\cG(\lambda,X^3+a)$ has no connected component with four vertices.
\end{proposition}

\begin{proof}
As before, if the graph $\cG(\lambda,X^3 +a)$
has a connected component with four vertices,
then there are only two possible cases as in Figure~\ref{fig:linear-4}.

For the first case when there is no fixed vertex in Figure~\ref{fig:linear-4},
we consider either
$$
y^2 = x^3 + a, \quad \lambda y^2 = -x^3+a, \quad x^2 = y^3 + a, \quad \lambda x^2 = -y^3+a,
$$
or
$$
y^2 = x^3 + a, \quad \lambda y^2 = -x^3 +a, \quad x^2 = -y^3 + a, \quad \lambda x^2 = y^3 +a.
$$
Then, noticing $x \ne \pm y$, we obtain $\lambda = -1, a = 0, x^5 =1$ and $x \ne 1$, $y=1/x$.
This contradicts with $a \ne 0$.
So, the first case cannot happen.

For the second case in Figure~\ref{fig:linear-4},
 noticing $ (x / y)^3 \ne \pm 1$ (due to $x \ne \pm y$ and $3 \nmid q-1$),
 we only need to consider the following four subcases:
\begin{itemize}
\item[(1)]  $x^2 = x^3 + a, \quad y^2 = -x^3 +a, \quad \lambda y^2 = y^3 + a, \quad \lambda x^2 = -y^3 + a$;

\item[(2)]  $x^2 = x^3 + a, \quad \lambda y^2 = -x^3 +a, \quad  y^2 = y^3 + a, \quad \lambda x^2 = -y^3 + a$;

\item[(3)]  $\lambda x^2 = x^3 + a, \quad  y^2 = -x^3 +a, \quad \lambda y^2 = y^3 + a, \quad  x^2 = -y^3 + a$;

\item[(4)]  $\lambda x^2 = x^3 + a, \quad \lambda y^2 = -x^3 +a, \quad  y^2 = y^3 + a, \quad  x^2 = -y^3 + a$.
\end{itemize}
By direct calculations,
from Cases (2) and (3) we obtain $\lambda = 1$, which contradicts with the assumption that $\lambda$ is non-square;
and Case (1) gives $\lambda^2 = -1, a = 0, x=1, y= \lambda$; Case (4) gives $\lambda^2 = -1, a = 0, x = \lambda, y =1$.
Hence,  the second case also cannot happen (due to  $\lambda$ non-square and $a \ne 0$).

Therefore, there is no connected component with four vertices in the graph $\cG(\lambda,X^3 +a)$ with $a \in \F_q^*$.
\end{proof}

From the above proof, we directly obtain:

\begin{proposition}
Assume $3 \nmid q-1$.
If $-1$ is a non-square element in $\F_q$,
then the graph $\cG(\lambda,X^3)$ has a connected component with four vertices if and only if $5 \mid q-1$ and $\lambda = -1$
\textup{(}in fact these four vertices are of the form  $x, -x, 1/x, -1/x$, where $x^5=1$ and $x \ne 1$, 
corresponding to the first case in Figure~\ref{fig:linear-4}\textup{)}.
Otherwise, if $-1$ is a square in $\F_q$,
then the graph $\cG(\lambda,X^3)$ has a connected component with four vertices if and only if $4 \mid q-1$ and $\lambda^2 = -1$
\textup{(}in fact these four vertices are $1, -1, \lambda, -\lambda$, corresponding to the second case in Figure~\ref{fig:linear-4}\textup{)}.
\end{proposition}

\subsection{Computations concerning connectedness}
\label{sec:cubic-com}

Recall that $p$ is an odd prime.
Here, we want to make some computations for the graphs $\cG(\lambda,X^3+a)$ over $\F_p$ concerning its connectedness
when $3 \nmid p-1$.
From Proposition~\ref{prop:cubic1}, we only need to consider the non-square elements $\lambda$ in $\N_p$, where $\N_p$ has been defined in \eqref{eq:Np}.
Our computations suggest that almost all the  graphs
$\cG(\lambda,X^3+a)$ ($\lambda$ runs over $\N_p$ and $a$ runs over $\F_p^*$)
are connected when $3 \nmid p-1$.

Let $C_3(p)$ (respectively, $U_3(p)$) be the number of connected (respectively, unconnected) graphs
among all the graphs $\cG(\lambda,X^3+a)$ when $\lambda$ runs over $\N_p$ and $a$ runs over $\F_p^*$,
and let $R_3(p)$ be the ratio of connected graphs, that is
$$
R_3(p) = \frac{C_3(p)}{C_3(p) + U_3(p)}.
$$
Further, let $M_3(p)$ be the maximum number of  connected components  of an unconnected graph counted in $U_3(p)$.

As the linear case, by Proposition~\ref{prop:cubic-com2}, Proposition~\ref{prop:cubic-com3} and \eqref{eq:Np}, we have
\begin{equation*}
U_3(p) \ge L(p),  \quad p > 5, \quad 3 \nmid p-1,
\end{equation*}
where $L(p)$ has been defined in \eqref{eq:Lp}.

From Table~\ref{tab:cubic-conn1}, we can see that almost all the graphs $\cG(\lambda,X^3+a)$ are connected,
and $U_3(p)$ is quite close to $L(p)$.
Moreover, the data suggest that each unconnected graph $\cG(\lambda,X^3+a)$ with $a \ne 0$ has exactly two connected components,
and its small connected component usually has exactly two vertices.

\begin{conjecture}  \label{conj:cubic-conn0}
Almost all the graphs $\cG(\lambda,X^3+a)$ over $\F_p$ are connected when $p$ satisfying $3 \nmid p-1$ goes to infinity.
\end{conjecture}

Our computations for all primes $p \le 3041$ satisfying $3 \nmid p-1$
shows that $U_3(p) = L(p)$ for any such prime $p \in [31,3041]$,
and for any $5 \le p \le 3041$ each unconnected graph $\cG(\lambda,X^3+a)$ with $a \ne 0$ over $\F_p$
has exactly two connected components.
Hence, we make the following conjectures.

\begin{conjecture}  \label{conj:cubic-conn1}
For any prime $p>31$ satisfying $3 \nmid p-1$, $U_3(p) = L(p)$.
\end{conjecture}

\begin{conjecture}  \label{conj:cubic-conn2}
For any prime $p \ge 5$ satisfying $3 \nmid p-1$, each unconnected graph $\cG(\lambda,X^3+a)$ with $a \ne 0$ over $\F_p$
has exactly two connected components.
\end{conjecture}

Conjecture~\ref{conj:cubic-conn1} suggests that when $p>31$ satisfying $3 \nmid p-1$,
if  a graph $\cG(\lambda,X^3+a)$ over $\F_p$ with $a \ne 0$
does not belong to the cases described in Propositions~\ref{prop:cubic-com2} and \ref{prop:cubic-com3},
then  it is a connected graph.

\begin{table}[!htbp]
\centering
\begin{tabular}{|c|c|c|c|c|c|}
\hline
$p$ & $C_3(p)$ & $U_3(p)$ & $L(p)$ & $M_3(p)$ & $R_3(p)$ \\ \hline

   41 &     188 &   8 &   8 & 2 & 0.95918 \\ \hline 

  107 &    2835 &  27 &  27 & 2 & 0.99056 \\ \hline 

  521 &   67470 & 130 & 130 & 2 & 0.99807 \\ \hline 

 1013 &  255782 & 254 & 254 & 2 & 0.99900 \\ \hline 

 1511 &  570403 & 377 & 377 & 2 & 0.99933 \\ \hline 

 2027 & 1026675 & 507 & 507 & 2 & 0.99950 \\ \hline 

 2531 & 1600857 & 633 & 633 & 2 & 0.99960 \\ \hline 

 3041 & 2309640 & 760 & 760 & 2 & 0.99967 \\ \hline 
\end{tabular}
\bigskip
\caption{Counting connected graphs in cubic case}
\label{tab:cubic-conn1}
\end{table}

\subsection{Hamiltonian cycles}

In light of Propositions~\ref{prop:perm-inout} and \ref{prop:cubic-indeg}, here we only consider the case when $3 \nmid p-1$.
Theorem~\ref{thm:perm-Ha} has confirmed the existence of Hamiltonian cycles
in connected components of the graph $\cG(\lambda,X^3+a)$ over $\F_p$ when $3 \nmid p-1$.
Here, we make some computations on counting Hamiltonian cycles of connected graphs.

Let $H_{31}(p)$ (respectively, $H_{32}(p)$) be the minimal (respectively, maximal) number of Hamiltonian cycles in a connected graph of the form $\cG(\lambda,X^3+a)$ over $\F_p$
($\lambda$ runs over $\N_p$ and $a$ runs over $\F_p^*$).
Then, let $H_3(p)$ be the average number of Hamiltonian cycles in these connected graphs.

By Corollary~\ref{cor:cubic-HT1}, a connected graph $\cG(\lambda,X^3+a)$ over $\F_p$ has no Hamiltonian cycle of Type $1$ when $p> 17$.
Let $R_{32}(p)$ (respectively, $R_{33}(p)$) be the ratio of such connected graphs having Hamiltonian cycles of Type $2$ (respectively, Type $3$)
over all the connected graphs (when $\lambda$ runs over $\N_p$ and $a$ runs over $\F_p^*$).
Let $A_{32}(p)$ (respectively, $A_{33}(p)$) be the average number of  Hamiltonian cycles of Type $2$ (respectively, Type $3$)
over all such connected graphs having Hamiltonian cycles of Type $2$ (respectively, Type $3$).

From Table~\ref{tab:cubic-hamilton}, we can see that there are many Hamiltonian cycles in a connected graph $\cG(\lambda,X^3+a)$ over $\F_p$,
whose amount grows rapidly with respect to $p$. 

\begin{table}[!htbp]
\centering
\begin{tabular}{|c|c|c|c|c|}
\hline
$p$ & $H_{31}(p)$ & $H_{32}(p)$ & $H_3(p)$ \\ \hline

11 &     1 &      11 &        4.18518 \\ \hline

17 &     1 &      64 &       19.20689 \\ \hline

23 &     3 &     330 &       91.65079 \\ \hline

29 &    24 &    2574 &      585.11702 \\ \hline

41 &   602 &  127573 &    26075.50256 \\ \hline

47 &  4444 &  923740 &   187351.93160 \\ \hline

53 & 35169 & 6920444 &  1262002.85498 \\ \hline

\end{tabular}
\bigskip
\caption{Counting Hamiltonian cycles in cubic case}
\label{tab:cubic-hamilton}
\end{table}

Tables~\ref{tab:cubic-HT2} and \ref{tab:cubic-HT3} suggest that although many connected graphs have Hamiltonian cycles of Types 2 and 3,
 these two types of Hamiltonian cycles occupy a small proportion when $p$ is large.
 This means that each connected graph $\cG(\lambda,X^3+a)$ over $\F_p$ is likely to have many types of Hamiltonian cycles.

\begin{table}[!htbp]
\centering
\begin{tabular}{|c|c|c|c|}
\hline
$p$  & $A_{32}(p)$ & $A_{32}(p)/H_3(p)$ & $R_{32}(p)$ \\ \hline

11 &   1.41666 & 0.33849 & 0.88888 \\ \hline

17 &   2.65853 & 0.13841 & 0.70689 \\ \hline

23 &   4.32394 & 0.04717 & 0.56349 \\ \hline

29 &   7.24647 & 0.01238 & 0.75531 \\ \hline

41 &  27.42268 & 0.00105 & 0.49743 \\ \hline

47 &  58.89082 & 0.00031 & 0.42329 \\ \hline

53 & 105.13149 & 0.00008 & 0.49395 \\ \hline

\end{tabular}
\bigskip
\caption{Hamiltonian cycles of Type 2 in cubic case}
\label{tab:cubic-HT2}
\end{table}

\begin{table}[!htbp]
\centering
\begin{tabular}{|c|c|c|c|}
\hline
$p$  & $A_{33}(p)$ & $A_{33}(p)/H_3(p)$ & $R_{33}(p)$ \\ \hline

11 &     2.50000 & 0.59734 & 0.81481 \\ \hline

17 &     7.69642 & 0.40071 & 0.96551 \\ \hline

23 &    31.16260 & 0.34001 & 0.97619 \\ \hline

29 &   119.66486 & 0.20451 & 0.98404 \\ \hline

41 &  1878.58549 & 0.07204 & 0.98974 \\ \hline

47 &  7503.64606 & 0.04005 & 0.98706 \\ \hline

53 & 37585.84218 & 0.02978 & 0.99546 \\ \hline
\end{tabular}
\bigskip
\caption{Hamiltonian cycles of Type 3 in cubic case}
\label{tab:cubic-HT3}
\end{table}

\section{Comments}

Assume that $f$ is a permutation polynomial over $\F_q$.
In Theorem~\ref{thm:perm-balance} we have shown that in the graph $\cG(\lambda,f)$,
along any Hamiltonian cycle of any connected component, we can get a balancing binary sequence.
Our computations in Sections~\ref{sec:linear-com} and \ref{sec:cubic-com} suggest that the graph $\cG(\lambda,f)$ is usually connected.
That is, using this way we can frequently obtain a balancing binary periodic sequence of period $q$.
Note that the balance property is one of the three randomness postulates about a binary sequence suggested by Golomb;
see  \cite[Chapter 5]{Gong}.
The other two postulates are called the run property and the correlation property.
If the types of graphs studied in this paper frequently yield a balancing sequence which also has good run property and correlation property,
then this gives a good way to construct pseudorandom number generators.

In addition, based on our computations the graph $\cG(\lambda,f)$ is likely to
be connected.
It will be interesting and also challenging to confirm this theorectically,
such as proving $\cG(\lambda,f)$ is connected for an infinite family of permutation polynomials.

When $Y^2=X^3+aX+b$ defines an elliptic curve over $\F_q$,
it is also interesting to investigate the relation between properties of  the graph $\cG(\lambda,X^3+aX+b)$
and the arithmetic of the corresponding elliptic curve.

In fact, the graphs studied in this paper are arised from quadratic twists (see \eqref{eq:equation}). 
One can generalize them to higher twists. 
For example, if $k \mid q -1$ and $\mu$ is not a $k$-th power in $\F_q$, one can study the graph generated by the equation
$$
(Y^k - f(X)) (\mu Y^k - f(X))  \cdots (\mu^{k-1} Y^k - f(X)) = 0. 
$$

More generally, for any $k$ polynomials $f_0, f_1, \ldots, f_{k-1} \in \F_q[X]$ and any positive integer $n$, 
one can  study the graph generated by the equation
$$
(Y^n - f_0(X)) (Y^n - f_1(X)) \cdots ( Y^n - f_{k-1}(X)) = 0. 
$$
Moreover, in this graph an edge $(x,y)$ has weight $i$ if $y^n = f_i(x)$.

\section*{Acknowledgements}

The authors want to thank the referees for their valuable comments. 
They also would like to thank Igor Shparlinski for stimulating discussions and useful comments.
For the research, Bernard Mans was partially supported
by the  Australian Research Council Grants DP140100118 and DP170102794, and  Min Sha by a
Macquarie University Research Fellowship and the Australian Research Council Grant DE190100888.

\end{document}